\documentclass{amsart}[a4paper]
\usepackage{graphicx} 
\usepackage[a4paper, total={6in, 9.5in}]{geometry}
\usepackage{mathtools, amsmath, amsthm, amssymb, amsfonts}

\newtheorem{thm}{Theorem}[section]
\newtheorem{cor}[thm]{Corollary}
\newtheorem{prop}[thm]{Proposition}
\newtheorem{lem}[thm]{Lemma}
\newtheorem*{thm*}{Theorem}

\theoremstyle{definition}
\newtheorem{defn}[thm]{Definition}

\theoremstyle{remark}
\newtheorem{rem}[thm]{Remark}
\newtheorem{ex}[thm]{Example}

\usepackage{nicefrac}

\usepackage{tikz-cd}

\usepackage{hyperref}

\usepackage{enumitem}

\numberwithin{equation}{section}

\newcommand{\boxc}[0]{\;\square\;}
\newcommand{\bb}[1]{\mathbb{#1}}

\newcommand{\id}[1]{\text{id}_{#1}}

\newcommand{\set}{\mathbf{Set}}
\newcommand{\sset}{\mathbf{sSet}}
\newcommand{\ssp}{\mathbf{sSet}^{\Delta^\mathrm{op}}}
\newcommand{\dset}{\mathbf{dSet}}
\newcommand{\dsp}{\mathbf{sSet}^{\Omega^\mathrm{op}}}
\newcommand{\sspsegal}{\ssp_{\mathrm{Segal}}}
\newcommand{\dspsegal}{\dsp_{\mathrm{Segal}}}
\newcommand{\dspcss}{\dsp_{\mathrm{CSS}}}
\newcommand{\sspcat}{\ssp_{\mathrm{Cat}}}
\newcommand{\dspop}{\dsp_{\mathrm{Op}}}

\newcommand{\op}{\mathbf{Op}}
\newcommand{\preop}{\mathcal{P}(\op)}
\newcommand{\cosk}{\mathrm{cosk}}
\newcommand{\std}[1]{\Delta[#1]} 
\newcommand{\horn}[2]{\Lambda^{#1}[#2]} 
\newcommand{\om}[1]{\Omega[#1]} 
\newcommand{\spn}[1]{\text{Sp}[#1]} 
\newcommand{\popd}[5]{#2 #1 #5\coprod_{#2 #1 #4}#3 #1 #4} 
\newcommand{\pop}[5]{\popd{#1}{#2}{#3}{#4}{#5}\longrightarrow #3 #1 #5} 

\newcommand{\Map}[2]{\text{Map}(#1,#2)}
\newcommand{\ho}{{\rm ho}}
\newcommand{\expd}[2]{#1^{\{#2\}}}

\newcommand{\mapd}[3]{\text{map}_{#1}(#2;#3)}
\newcommand{\mapds}[4]{\mapd{#1}{{#2}_1,\ldots, {#2}_{#3}}{#4}}

\title[A new model structure on dendroidal spaces]{A new model structure on dendroidal spaces\\ for the theory of $\infty$-operads}
\author[J. Candeias]{João Candeias}
\address{João Candeias: Departament de Matem\`atiques i Inform\`atica, Universitat de Barcelona (UB). Gran Via de les Corts Catalanes 585. 08007, Barcelona, Spain}
\author[J.J. Gutiérrez]{Javier J. Guti\'errez}
\address{Javier J. Guti\'errez: Departament de Matem\`atiques i Inform\`atica, Universitat de Barcelona (UB). Gran Via de les Corts Catalanes 585. 08007, Barcelona, Spain}

\begin{document}

\begin{abstract}
We introduce a new model structure on the category of dendroidal spaces, designed to provide a further model for the homotopy theory of $\infty$\nobreakdash-operads. This model is directly analogous to a recent construction on the category of simplicial spaces
by Moser and Nuiten, and can be seen as its dendroidal counterpart. In our new model structure, the fibrant objects are the dendroidal Segal spaces, while the cofibrations form a subclass of those in the Segal space model structure. The weak equivalences between fibrant objects are precisely the Dwyer--Kan equivalences, and the fibrations between them are the isofibrations. We prove that this model structure is Quillen equivalent to the Cisinski–Moerdijk model structures on dendroidal sets and dendroidal spaces, thereby establishing a compatible extension of the theory of $\infty$-categories
to the operadic setting.
\end{abstract}

\maketitle

\section{Introduction}
Dendroidal sets and dendroidal spaces are presheaf and simplicial presheaf categories, respectively, on the category $\Omega$ of trees. Both categories were introduced by Moerdijk--Weiss~\cite{Moerdijk_2007} and exploited by Cisinski--Moerdijk~\cite{Cisinski_Moerdijk_1,Cisinski_Moerdijk_2} in order to provide combinatorial models for the notion of operad up to homotopy or $\infty$-operad. They generalize the theory of $\infty$-categories developed by Rezk, Joyal and Lurie, among others, in particular the models of quasi-categories and complete Segal spaces.

The idea behind both of these models for $\infty$-categories begins with the observation that a category can be described as a simplicial set $X$ satisfying the property that, for every spine inclusion of simplices $\spn{n}\hookrightarrow \std{n}$, the induced map
$$
   \sset(\std{n},X)\to \sset(\spn{n},X)
$$
is a bijection for every $n\ge 0$, where $\sset$ denotes the category of simplicial sets. This means that there is a unique way of composing $n$ consecutive maps.

The quasi-category model structure $\sset_{\text{qCat}}$ on simplicial sets generalizes this idea by requiring only that the map above be surjective. Replacing simplicial sets with simplicial spaces, that is, functors from $\Delta^{\rm op}$ to simplicial sets, we can obtain the complete Segal space model structure $\ssp_{\text{CSS}}$ on simplicial spaces. More specifically, the category of simplicial spaces comes equipped with a Reedy model structure enriched over the classical Kan--Quillen model structure on simplicial sets, so that now we have induced maps of \emph{simplicial sets}
$$
\Map{N[n]}{X}\to \Map{\spn{n}}{X}
$$
for every $n\ge 0$, where $N[n]$ is the discrete nerve of the category with $n+1$ objects and $n$ composable arrows, 
and the condition imposed is that these maps are weak equivalences.
Reedy fibrant simplicial spaces satisfying this last condition are called \emph{Segal spaces}, and in these it makes sense to speak of spaces of objects, arrows, pairs of composable arrows and their composition, and so on. There is a model structure $\ssp_{\text{Segal}}$ on the category of simplicial spaces, which is a localization of the Reedy model structure, and for which the cofibrations are the Reedy cofibrations and the fibrant objects are the Segal spaces. 
However, to obtain Rezk's model one must perform an additional step called completion, which ensures that the resulting object is a \emph{complete Segal space}. By further localizing the previous model, we obtain the model structure $\ssp_{\text{CSS}}$, which has for cofibrations the monomorphisms, and the fibrant (immediately also cofibrant) objects are the complete Segal spaces.

The need for completion can be motivated as follows. For a general Segal space $X$, its space of objects~$X_0$ is somewhat related to the subspace of isomorphisms $X_{\text{hoeq}}\subset X_1$. In particular, if $x,y\in X_{0,0}$ are vertices that lie in the same path component of $X_0$, then there necessarily exists an isomorphism $f\colon x\to y$. 
However, the converse need not hold, that is, $X_0$ may fail to capture the full (higher) homotopical information encoded in the space of isomorphisms. In particular, it is possible for vertices in distinct components of $X_0$ to be connected by an isomorphism in $X_{\text{hoeq}}$.

This mismatch causes difficulties when defining weak equivalences. In $\ssp_{\text{Segal}}$, the weak equivalences between Segal spaces $f\colon X\to Y$ coincide with the Reedy weak equivalences, that is, termwise weak equivalences
$$
    f_n \colon X_n\longrightarrow Y_n
$$
for every $n\ge 0$. By the Segal condition, it suffices to require that $f_n$ is a weak equivalence only for $n=0,1$. However, if we want to model equivalences of $\infty$-categories, that is, essential surjectivity and fully faithfulness, it is enough to require that, for $n=0$, the map
$$   \pi_0(X_0)_{/\sim}\longrightarrow\pi_0(Y_0)_{/\sim}
$$
is a surjection, where $\sim$ is the relation coming from the space $X_{\text{hoeq}}\subset X_1$, and that for $n=1$ the square
\begin{equation*}
    \begin{tikzcd}
    X_1 \arrow[d] \arrow[r] & Y_1 \arrow[d] \\
    X_0\times X_0 \arrow[r] & Y_0\times Y_0
    \end{tikzcd}
\end{equation*}
is a homotopy pullback. A map satisfying these conditions is called a \emph{Dwyer--Kan equivalence}. Note that relaxing the demand that $f_0$ be a weak equivalence to asking that $\pi_0(X_0)_{/\sim}\to\pi_0(Y_0)_{/\sim}$ is surjective is indeed the sensible thing to do, given the lack of information in $X_0$ and $Y_0$ regarding isomorphism classes (and, more generally, the higher homotopy of the space of isomorphisms).

Rezk's process of completion augments $X_0$ with this missing data, turning it into a sort of ``moduli space'' of objects that is canonically equivalent (up to homotopy) to the space of isomorphisms of $X$. In particular, the Dwyer--Kan equivalences $f\colon X\to Y$ between complete Segal spaces are precisely the Reedy weak equivalences, so that the distinction mentioned above vanishes.
Moreover, for any Segal space $X$, the completion map $X\to \widetilde{X}$ (a fibrant replacement in $\ssp_{\textrm{CSS}}$) can be chosen to be a Dwyer--Kan equivalence, showing that the Dwyer--Kan equivalences are precisely those maps of Segal spaces that become weak equivalences in $\ssp_{\textrm{CSS}}$.

Recently, a new equivalent model for $\infty$-categories has been introduced by Moser and Nuiten in~\cite{moser2024}, again based on Segal spaces. Instead of further localizing to obtain complete Segal spaces, this approach builds from scratch, using a construction described in \cite{guetta2023}, a model structure $\sspcat$ on the category of simplicial spaces. In this model, the fibrant objects are precisely the Segal spaces (not necessarily complete), but there are fewer cofibrations than in the Segal model structure. The cofibrant objects are those simplicial spaces $X$ such that $X_0$ is weakly equivalent to a set, so that the fibrant-cofibrant objects are the Segal spaces whose space of objects is, up to homotopy, a set. Moreover, the weak equivalences between Segal spaces in this model are precisely the Dwyer--Kan equivalences.

Thus, this provides another solution to the compatibility problem described above: rather than adding all the homotopical information to $X_0$, as in the passage to complete Segal spaces, one instead removes it entirely.
This appears to parallel the contrast between the set-theoretic and univalent perspectives in type theory: in the former there is only a single mode of identification, namely equality, whereas in the latter the identity type is maximally rich, encompassing all equivalences. Furthermore, this phenomenon already arises in the definition of 1-categories within homotopy type theory, giving rise to the notions of precategory, univalent category, and strict category (as defined in \cite{AHRENS_2015}), which are essentially the 1-categorical analogues of Segal spaces, complete Segal spaces, and the cofibrant Segal spaces in $\sspcat$, respectively. It should be noted that, in general, univalent categories seem to be a better-behaved notion when working in univalent foundations, as for instance the statement that an essentially surjective and fully faithful functor is an equivalence always holds regardless of the validity of the axiom of choice, whereas for strict categories the former and the latter are equivalent (see \cite{AHRENS_2015}). This suggests that the univalent point of view might be more convenient for $\infty$-categories as well. If, however, one wishes to develop their $\infty$-category theory in set-theoretic foundations, it might be the case that $\sspcat$ ends up being technically more convenient: even though complete Segal spaces better align with the simplicial homotopy theory used to define them, it is in practice easier to turn the space of objects of a Segal space into a set (to get a cofibrant-fibrant object in $\sspcat$) than it is to perform the completion step (to get a complete Segal space).

In order to generalize the above ideas from categories to operads, Moerdijk and Weiss introduced the category of trees $\Omega$ in \cite{Moerdijk_2007}. This category contains $\Delta$ as a full subcategory, and its presheaf category naturally provides the setting in which the nerves of operads live. Generalizing what happens in the categorical case, operads can be redefined as dendroidal sets $X$ such that, for every tree $T$, the maps
$$
\dset(T,X)\longrightarrow \dset(\spn{T},X)
$$
are isomorphisms, where $\spn{T}$ denotes an appropriate notion of spine of a tree. Following the same principle as before gives rise to two established models of $\infty$-operads: quasi-operads and complete dendroidal Segal spaces. Moreover, these are compatible with their simplicial counterparts, in the sense that slicing over the dendroidal set or dendroidal space represented by the trivial tree $\eta$ recovers the model structures of quasi-categories and complete Segal spaces.

The main goal of this paper is to establish a new model structure $\dsp_{\rm Op}$ on the category of dendroidal spaces, which generalizes and is compatible with the model structure $\ssp_{\rm Cat}$. In this model, the fibrant objects are the dendroidal Segal spaces, and the cofibrations form a subclass of those in the Segal space model structure.
Our main result can be summarized as follows (see Theorem~\ref{thm:main_existence} and Theorem~\ref{thm:main_quilleneq}):

\begin{thm}\label{thm:main}
There exists a cofibrantly generated model structure on $\dsp$, which we denote by $\dsp_{\rm Op}$, such that:
\begin{itemize}
\item[{\rm (i)}] The cofibrations are the normal monomorphisms $f\colon X\to Y$ such that there exists a set $R$ and a weak equivalence $X_{\eta}\coprod R\to Y_{\eta}$ extending $f_{\eta}$.
\item[{\rm (ii)}] The fibrant objects are the dendroidal Segal spaces.
\item[{\rm (iii)}] The weak equivalences between fibrant objects are the Dwyer--Kan equivalences.
\item[{\rm (iv)}] The fibrations between fibrant objects are the isofibrations.
\end{itemize}
Moreover, this model structure is Quillen equivalent to the Cisinski--Moerdijk model structures on dendroidal sets and dendroidal spaces, thereby providing yet another model for the theory of $\infty$\nobreakdash-operads.
\end{thm}

This result closely parallels the main theorem of~\cite{moser2024}, with adaptations reflecting the transition from the simplicial to the dendroidal setting. In particular, $\dspop$ is homotopically enriched over the model structure $\sspcat$ constructed in~\cite{moser2024}, and slicing the former over the dendroidal space $\Omega[\eta]$ recovers the latter.
While the strategy of the proof in the dendroidal setting is inspired by the simplicial case, new difficulties arise that require careful treatment, and several arguments must be refined or extended to accommodate the dendroidal situation. In this way, the present work provides a genuine extension of the methods of \cite{moser2024}, establishing their full dendroidal analogue.

\bigskip
\noindent\textbf{Organization of the paper.} The paper is organized as follows.
Section~2 recalls the necessary background on simplicial and dendroidal sets and spaces, together with the categorical and model-categorical constructions that will be used in the sequel.
In Section~3, we introduce dendroidal Segal spaces and develop their fundamental homotopy-theoretic properties, including the notions of Dwyer--Kan equivalence and isofibration.
Section~4 is devoted to the construction of the model structure 
$\dsp_{\rm Op}$ and the verification of the hypotheses required for its existence.
In Section~5, we establish the Quillen equivalences with the Cisinski--Moerdijk model structures on dendroidal sets and dendroidal spaces, thereby completing the comparison with existing models of $\infty$-operads. Finally, Section 6 applies the comparison results of the previous sections to homotopy limits and colimits. We show that, for suitable diagrams, homotopy limits and colimits computed in $\dspop$ coincide with those computed in $\dspcss$, allowing one to avoid the completion step in many cases.

\bigskip
\noindent\textbf{Acknowledgements.}
The authors would like to thank Joost Nuiten for his availability in explaining and discussing details of several of the proofs in \cite{moser2024}.
This work was supported by MCIN/AEI under I+D+i grant PID2024-155646NB-I00 and by the Departament de Recerca i Universitats de la Generalitat de Catalunya with
reference 2021 SGR 00697. The first author is supported by the PhD grant 2023.02597.BD awarded by Fundação para a Ciência e Tecnologia.

\section{Preliminaries}
In this section, we introduce the fundamental notation and review the essential facts concerning simplicial and dendroidal sets and spaces. For further background and a more comprehensive treatment of these notions, we refer the reader to \cite{Cisinski_Moerdijk_1,Cisinski_Moerdijk_2,Heuts2022,Moerdijk_2007}.

\subsection{Trees} Let $\Omega$ denote the category of trees of Moerdijk--Weiss~\cite{Moerdijk_2007}. In this category there are a few basic trees which we will use repeatedly. First, there is the \emph{trivial tree} $\eta$, which is the only tree with zero vertices and one edge. Then for each $k\in\bb{N}$ there is the \emph{$k$-corolla} $C_k$, which is the tree with one vertex and $k$ leaves. Any tree can be constructed in an essentially unique way from corollas by \emph{grafting}: if $T$ and $S$ are trees and $e$ is an edge of $S$, then $S\circ_e T$ is the tree obtained by identifying the root of $T$ with the leaf $e$ of $S$. The two-sided identity for this operation is then of course the trivial tree $\eta$. Finally, any tree $T$ generates an operad $\Omega(T)$, which is free with generators given by its decomposition in corollas. Arrows $T\to S$ in $\Omega$ are maps of operads $\Omega(T) \to \Omega(S)$.

\subsection{Simplicial and dendroidal sets and spaces} Let $\sset$ denote the category of presheaves on $\Delta$ and $\dset$ the category of presheaves on $\Omega$. These are the categories of simplicial sets and dendroidal sets, respectively. We denote the representable presheaves in $\sset$ by $\std{n}$, for every $n\ge 0$, that is, $\std{n}=\Delta(-,[n])$. Similarly, the representable presheaves in $\dset$ are denoted by $\om{T}$ for every tree $T$ in $\Omega$. 
Oftentimes we write $T$ instead of $\om{T}$ to avoid cumbersome notation; this mild abuse is harmless thanks to the Yoneda lemma. We also use the standard notation $\partial\std{n}$ and $\horn{k}{n}$ for the boundary and $k$-horn of the representable $\std{n}$, and $\partial T$ and $\spn{T}$ for the boundary and spine of the representable dendroidal set $\om{T}$. Simplicial (resp. dendroidal) sets satisfying the inner horn filling condition will be called quasi-categories (resp. quasi-operads).

The inclusion $i:\Delta\to\Omega$ induces an adjoint triple between the corresponding presheaf categories
\begin{equation*}
    \begin{tikzcd}
    \sset \arrow[r, "i_!", bend left=49] \arrow[r, "i_*"', bend right=49] & \dset. \arrow[l, "i^*"']
    \end{tikzcd}
\end{equation*}
Under the identifications $\sset \cong \dset_{/\om{\eta}}$ and $\dset \cong \dset_{/*}$, the adjoint triple above can be interpreted as the pullback functor along the unique map $\om{\eta} \to *$, together with its left and right adjoints. The same description remains valid after replacing $\set$ with $\sset$; explicitly,
$$
\begin{tikzcd}
    \ssp \arrow[r, "i_!", bend left=49] \arrow[r, "i_*"', bend right=49] &
    \dsp \arrow[l, "i^*"']
\end{tikzcd}
$$
is again the adjoint triple associated to the equivalence
$$
    \ssp \cong (\dsp)_{/\om{\eta}},
$$
where $\om{\eta}$ is now viewed as a discrete dendroidal space. In particular, this yields a natural isomorphism
$$
    \om{\eta} \times X \cong i_! i^* X.
$$

\subsubsection{Vertical and horizontal directions} An object of $\dsp$ can be viewed in three equivalent ways: as a simplicial dendroidal set, as a dendroidal simplicial set, or as a presheaf on $\Omega \times \Delta$. One may picture such an object as a two–dimensional array of sets, with the $\Omega$–structure acting horizontally and the $\Delta$-structure acting vertically. Although the choice of which direction is designated horizontal or vertical is arbitrary, we fix it for the remainder of this text, as it helps to articulate arguments more clearly.

\subsubsection{Discrete and constant dendroidal spaces} There are inclusions of categories
$$
    \sset \xhookrightarrow{i_v} \dsp \xhookleftarrow{i_h} \dset
$$
which we refer to as the \emph{vertical} and \emph{horizontal} inclusions. According to the conventions just introduced, the vertical inclusion $i_v$ embeds a simplicial set as a dendroidal space that is horizontally constant, while the horizontal inclusion $i_h$ embeds a dendroidal set as a dendroidal space that is vertically discrete. The relationships between these various presheaf categories may be summarized in the following diagram:
\begin{equation}\label{eq:pshv}
    \begin{tikzcd}
                                 &                                                                     & \Delta \arrow[d, "y_\Delta"] \\
                                 & \mathbf{\set} \arrow[r, "\delta_\Delta"] \arrow[d, "\delta_\Omega"] & \sset \arrow[d, "i_v"]       \\
    \Omega \arrow[r, "y_\Omega"] & \dset \arrow[r, "i_h"] \arrow[d, "i^*"]                             & \dsp \arrow[d, "i^*"]        \\
    \Delta \arrow[r, "y_\Delta"] & \sset \arrow[r, "i_h"']                                             & \ssp                        
    \end{tikzcd}
\end{equation}
The functors appearing in the two squares are precomposition functors of the form $\set^{J} \to \set^{I}$ induced by a functor $I \to J$. By standard properties of presheaf categories, each such functor admits both a left and a right adjoint. We single out the left adjoint of the horizontal inclusion, which is the functor $(\pi_0)_*$ given by post-composition by $\pi_0:\sset\to \set$.

\subsubsection{Representables and simplicial enrichment}
Since $\dsp$ is a presheaf category, it is in particular cartesian closed, with limits and colimits computed pointwise. For an object $(T,n) \in \Omega \times \Delta$, the presheaf $P$ represented by $(T,n)$ is given by
$$
    P_{S,m}
    = (\Omega \times \Delta)\bigl((S,m),(T,n)\bigr)
    = \Omega(S,T) \times \Delta(m,n)
    = \Omega[T]_S \times \Delta[n]_m.
$$
Thus $P$ is naturally isomorphic to $i_h(\Omega[T]) \times i_v(\Delta[n])$, where the product is taken in $\dsp$. From now on we omit explicit mention of $i_h$ and freely regard dendroidal sets as (discrete) objects of $\dsp$
The vertical inclusion $i_v : \sset \to \dsp$ together with cartesian-closedness of $\dsp$ give it the structure of a simplicially tensored and cotensored category, so that we get a two-variable adjunction
$$
    \dsp(X \times i_v(S), Y)
    \cong
    \sset(S, \Map{X}{Y})
    \cong
    \dsp\bigl(X,\, Y^{\,i_v(S)}\bigr),
$$
From now on we write the simplicial tensor $X\times i_v(S)$ as $X\cdot S$. Thus in particular
$$
    \Map{X}{Y}_n
    = \dsp(X \cdot \Delta[n],\, Y).
$$
Note also that representables may be written as $T \cdot \Delta[n]$, and thus
$$
    \Map{T}{X} \cong X_T.
$$
In particular, for any dendroidal subset of $T$, for example $\partial T$, we write $X_{\partial T}$ for $\Map{\partial T}{X}$.

All of the above constructions carry over verbatim to simplicial spaces, and they are compatible with the adjunction $i_! \dashv i^*$. In fact, $i^*(X \cdot S) \cong i^*(X) \cdot S$ and more importantly $i_!(X \cdot S)\cong i_!(X) \cdot S$, so that the adjunction $i_! \dashv i^*$ can be promoted to an enriched one:
$$
    \Map{i_! X}{Y}\cong\Map{X}{i^* Y}.
$$
As a consequence, we do not distinguish notationally between the dendroidal and simplicial mapping space constructions, although we will indicate the ambient category when necessary. Notably, for simplicial spaces $X$ and $Y$ we have
$$
    \Map{i_! X}{i_! Y} \cong\Map{X}{Y},
$$
which further justifies this mild abuse of notation.

\subsubsection{Skeleta and coskeleta} Let $\Delta_{\leq n}$ denote the full category of $\Delta$ whose objects are $[m]$ for $m\le n$. Similarly, let $\Omega_{\leq n}$ denote the full subcategory of $\Omega$ consisting of trees with $n$ or less vertices. The inclusions $\Delta_{\leq n} \to \Delta$ and $\Omega_{\leq n} \to \Omega$ induce restriction functors between the corresponding (simplicial) presheaf categories, and these restriction functors admit both left and right adjoints given by left and right Kan extension. We will be particularly interested in the dendroidal case with $n = 0$, in which the functor
$$
    \dsp \longrightarrow \sset^{\Omega_{\leq 0}^{\mathrm{op}}} \cong \sset
$$
is simply evaluation at~$\eta$. Its left and right adjoints will be denoted by $\operatorname{sk}_\eta$ and $\operatorname{cosk}_\eta$, respectively. A straightforward computation shows that, for a simplicial set $S$, these are given by
$$
    (\operatorname{sk}_\eta S)_T =
    \begin{cases}
        S & \text{if $T$ is linear,} \\
        \emptyset & \text{otherwise,}
    \end{cases}
$$
and
$$
    (\operatorname{cosk}_\eta S)_T
    =\!\!\!
    \prod_{\Omega(\eta,\,T)} S.
$$

\subsubsection{Normal monomorphisms} A monomorphism of dendroidal sets $f:X\xhookrightarrow{} Y$ is said to be \emph{normal} if for each tree $T$, the group $\text{Aut}(T)$ acts freely on $Y_T\setminus f(X_T)$, and a dendroidal set $X$ is said to be normal if the monomorphism $\emptyset\to X$ is normal. These monomorphisms admit a canonical ``skeletal'' filtration, which is a factorization of $f$ as a composition
\begin{equation*}
    X=X_0\xhookrightarrow{i_1} X_1\xhookrightarrow{i_2} X_2\xhookrightarrow{}\ldots \xhookrightarrow{} \varinjlim_{k\in\bb{N}}X_k=Y
\end{equation*}
where each $i_k$ is a pushout of a coproduct of maps of the form $\partial T\xhookrightarrow{}T$ (which are themselves normal) for trees $T$ with $k$ vertices. The class of normal monomorphisms is the smallest saturated class containing all boundaries of trees $\partial T\to T$.
Likewise, a morphism of dendroidal \emph{spaces} is said to be normal if it induces a normal monomorphism of dendroidal sets at each simplicial degree.

\subsection{Categories, operads and their nerves} 
We denote the free-living isomorphism category as $I[1]$. We write $j_! : \mathbf{Cat} \to \mathbf{Op}$ for the full and faithful inclusion of the category of small categories into the category of operads. We have a commutative square of adjoint functors (with left adjoints oriented from left to right and from top to bottom):
\begin{equation*}
    \begin{tikzcd}
    \sset \arrow[d, "i_!"', shift right] \arrow[r, "\tau", shift left]   & \mathbf{Cat} \arrow[d, "j_!"', shift right] \arrow[l, "N", shift left]  \\
    \dset \arrow[r, "\tau_d", shift left] \arrow[u, "i^*"', shift right] & \mathbf{Op}. \arrow[u, "j^*"', shift right] \arrow[l, "N_d", shift left]
    \end{tikzcd}
\end{equation*}
The horizontal pairs are the usual nerve-realization adjunctions obtained by Kan extension associated to the standard cosimplicial object in $\mathbf{Cat}$ and the codendroidal object in $\mathbf{Op}$, respectively.

Besides the commutation isomorphisms for the left/right adjoints $j_!\tau\cong\tau_di_!$ and $Nj^*\cong i^*N_d$, we also have the compatibility $N_d j_!\cong i_!N$. The identity $j^* \tau_d \cong \tau\, i^*$ does not hold in general, although it becomes valid when restricted to the full subcategory of quasi-operads.
Composing the square above with the bottom square of~\eqref{eq:pshv} yields the (discrete) nerve functors from operads (resp.\ categories) to dendroidal (resp.\ simplicial) spaces, which we denote with the same symbols, together with left adjoints of the form $\tau \circ \pi_0$.

\subsection{Tensor product of dendroidal sets and spaces}

The category $\dset$ carries a tensor product functor
$$
    \otimes\colon\dset \times \dset \longrightarrow \dset
$$
characterized as the essentially unique bifunctor that preserves colimits in each variable and sends a pair of representables $(T,S)$ to the nerve of the Boardman--Vogt tensor product of the corresponding operads:
$$
    T \otimes S = N_d\bigl(\Omega(T) \otimes_{\mathrm{BV}} \Omega(S)\bigr).
$$
This tensor product is symmetric and has the presheaf represented by the trivial tree $\eta$ as a unit. Extending the construction to dendroidal spaces, we define
$$
    \otimes\colon \dsp \times \dsp \longrightarrow \dsp
$$
as the pointwise tensor product in $\Delta^{\rm op}$, that is,
$$
    (X \otimes Y)_{T,n}
    =
    \bigl(X_{-,n} \otimes Y_{-,n}\bigr)_T.
$$
This product is again symmetric and preserves colimits in each variable, since all colimits are computed termwise. Consequently, it induces a two-variable adjunction
\begin{equation}\label{eq:adj_otimes}
    \dsp(X \otimes Y,\, Z)
    \cong
    \dsp\bigl(X,\, Z^{\{Y\}}\bigr)
    \cong
    \dsp\bigl(Y,\, Z^{\{X\}}\bigr),
\end{equation}
where
$$
    \bigl(Z^{\{Y\}}\bigr)_{T,n}
    =
    \dsp\bigl(Y \otimes (T \cdot \Delta[n]), Z\bigr).
$$
Moreover, since $(X \otimes Y)_{-,n}$ depends only on $X_{-,n}$ and $Y_{-,n}$, and since $(X \cdot S)_{-,n}\cong\coprod_{S_n} X_{-,n}$, 
the fact that $\otimes$ preserves colimits in each variable implies that it also preserves the simplicial tensor in each variable. Thus we obtain an enriched adjunction
$$
    \Map{X \otimes Y}{Z}
    \cong
    \Map{X}{Z^{\{Y\}}}
    \cong
    \Map{Y}{Z^{\{X\}}},
$$
and in particular $(Z^{\{X\}})_{\eta}=\Map{X}{Z}$.
Finally, we note that $i^*(Y^{\{\,i_! X\,\}})\cong (i^*Y)^X$, where the latter is the exponential object in $\ssp$.

\subsubsection{Monoidality and weak enrichment}

The tensor product $\otimes$ does not form part of a closed symmetric monoidal structure on $\dset$ or $\dsp$. In fact, it satisfies all required properties except for the existence of associator isomorphisms. Nevertheless, it extends to endow $\dset$ with the weaker structure of a colax symmetric monoidal category \cite[Definition~6.3.1]{Heuts2015}. Concretely, this consists of a category together with a specified $n$-fold tensor product funtor $\otimes_n$ for each $n$, together with natural maps
$$
\begin{tikzcd}
\otimes_n(X_1,\ldots,X_{i-1},\otimes_m(Y_1,\ldots,Y_m),X_{i+1},\ldots,X_n) \arrow[d] \\
\otimes_{m+n-1}(X_1,\ldots,X_{i-1},Y_1,\ldots,Y_m,X_{i+1},\ldots,X_n)
\end{tikzcd}
\label{eq:associator}
$$
satisfying the evident coherence conditions for composites of such maps, and such that each $\otimes_n$ is part of an $(n+1)$-variable adjunction. Moreover, the restricted adjunction
$$
\dsp(X\otimes i_!M,Y)\cong\dsp(X,\expd{Y}{i_!M})\cong\ssp(M,i^*(\expd{Y}{X}))
$$
provides $\dsp$ with a hom-object $i^*(\expd{Y}{X})$ internal to $\ssp$. However, due to the same failure of associativity, this adjunction does not extend to an $\ssp$-enriched adjunction, and hence the resulting enrichment is neither tensored nor cotensored. A weaker statement still holds: $\dsp$ is \emph{weakly} enriched, tensored and cotensored over $\ssp$ in the sense of \cite[Section~3.5]{Heuts2015} (a more accurate terminology would perhaps be “enriched and weakly tensored/cotensored”). In brief, this weak enrichment amounts to the existence of natural maps
\begin{equation}
\alpha_{M,N,X} :i_!(M\times N)\otimes X\to i_!M\otimes(i_!N\otimes X)
\label{eq:weak_enrichment}
\end{equation}
satisfying coherence conditions, including a directed form of Mac Lane’s pentagon. Note that in the presence of a genuine tensor or cotensor, these maps exist and are isomorphisms.

\subsection{Segal pre-operads}\label{sec:segal_preop}

A \emph{Segal pre-operad} is a dendroidal space $X$ such that $X_\eta$ is a discrete simplicial set. These form a full subcategory of $\dsp$, which we denote by $\preop$.
The category $\preop$ can be identified with the category of presheaves on the localization of $\Omega \times \Delta$ obtained by inverting all morphisms of the form $(id_\eta,\alpha)\colon(\eta,[m])\to (\eta,[n])$. Writing this localization as $(\eta \times \Delta)^{-1}(\Omega \times \Delta)$, the localization functor
$$
\gamma\colon \Omega \times \Delta \longrightarrow (\eta \times \Delta)^{-1}(\Omega \times \Delta)
$$
induces a functor on presheaf categories
\begin{equation*}
    \gamma^*\colon \preop \longrightarrow \dsp,
\end{equation*}
which admits both a left adjoint $\gamma_!$ and a right adjoint $\gamma_*$. The functor $\gamma_*$ is given by the pullback
$$
    \begin{tikzcd}
        \gamma_* X \arrow[r] \arrow[d, hook] & \text{cosk}_\eta X_{\eta,0} \arrow[d, hook] \\
        X \arrow[r]                          & \text{cosk}_\eta X_\eta
    \end{tikzcd}
$$
and can be described as the maximal dendroidal subspace of $X$ which is discrete at $\eta$. The existence of the adjunction follows from the observation that the image of any map from a pre-operad $P$ to a dendroidal space $X$ is necessarily contained in the subspace $\gamma_*X$. The left adjoint $\gamma_!$ admits a dual description.

\subsection{Generalized Reedy categories and model structures}

Both $\Delta$ and $\Omega$ are examples of generalized Reedy categories in the sense of \cite{Berger_2010}. This notion extends the classical concept of a Reedy category. If $\mathbb{R}$ is a generalized Reedy category and $\mathcal{E}$ is a model category, then the functor category $\mathcal{E}^\mathbb{R}$ admits a Reedy-type model structure whenever $\mathcal{E}$ is such that either the projective or the injective model structure exists on $\mathcal{E}^{\mathrm{Aut}(\mathbb{R})}$ for every $r\in\mathbb{R}$; see \cite[Theorem~1.6]{Berger_2010}. In our setting we consider the Reedy model structure on $\dsp$ inherited from the model structure on $\sset$, using the generalized Reedy structure on $\Omega$.

Central to the definition of the model structure are the latching and matching morphisms, defined via appropriate (co)limits. For a map of dendroidal spaces $f\colon X\to Y$ and a tree $T$, the relative matching map at $T$ is
$$
    X_T \to X_{\partial T}\times_{Y_{\partial T}} Y_T,
$$
where $X_{\partial T}=\Map{\partial T}{X}$ as defined previously.

Since $\Omega$ is a generalized Reedy category, it determines a \emph{Reedy model structure} on $\dsp$, denoted $\dsp_R$, specified as follows:
\begin{itemize}
    \item[{\rm (i)}] The weak equivalences are the maps $w\colon X\to Y$ such that each $w_T\colon X_T\to Y_T$ is a weak equivalence of simplicial sets.
    \item[{\rm (ii)}] The fibrations are the maps $p\colon X\to Y$ for which the matching maps
    $$
        X_T \to X_{\partial T}\times_{Y_{\partial T}} Y_T
    $$
    are Kan fibrations for all trees $T$.
    \item[{\rm (iii)}] The cofibrations are the normal monomorphisms.
\end{itemize}

Using the adjunction $X\cdot(-)\dashv\Map{X}{-}\colon\sset\rightleftarrows\dsp$, one checks that a map $p\colon X\to Y$ is a Reedy fibration if and only if it has the right lifting property with respect to the maps
\begin{equation}
    \pop{\cdot}{\partial T}{T}{\horn{k}{n}}{\std{n}},
\label{eq:gen_triv_cof_Reedy}
\end{equation}
and it is a trivial Reedy fibration if and only if it lifts against
\begin{equation}
    \pop{\cdot}{\partial T}{T}{\partial\std{n}}{\std{n}}.
\label{eq:gen_cof_Reedy}
\end{equation}
These families are therefore generating sets of (trivial) Reedy cofibrations.

\section{Dendroidal Segal spaces}
In this section we recall the notion of dendroidal Segal spaces, which play for $\infty$-operads the role that Rezk’s Segal spaces play for $\infty$-categories. We begin by reviewing the Segal condition for dendroidal spaces and recalling its basic consequences. 
We then develop several fundamental homotopy-theoretic notions in this setting, most notably Dwyer–Kan equivalences and isofibrations, which will form the weak equivalences and fibrations between fibrant objects in the model structures constructed later on. For background on dendroidal Segal spaces, their model structures and foundational properties we refer the reader to the work of Cisinski–Moerdijk~\cite{Cisinski_Moerdijk_2} and to the detailed treatment in \cite{Heuts2022}.

\begin{defn}
    Let $X$ be a Reedy fibrant dendroidal space. We say that $X$ is a \emph{dendroidal Segal space} if for each tree $T$ the Segal map
    \begin{equation}\label{eq:segalmap}
        X_T\longrightarrow X_{\spn{T}}=\Map{\spn{T}}{X}
    \end{equation}
    induced by the inclusion $\spn{T}\to T$ is a weak equivalence in $\sset$.
\end{defn}

\begin{rem}
    For every Reedy fibrant dendroidal space $X$, the map \eqref{eq:segalmap} is already a fibration. Thus $X$ is Segal if and only if \eqref{eq:segalmap} is a \emph{trivial} fibration.
\end{rem}

This definition generalizes Rezk’s notion of a simplicial Segal space~\cite{Rezk_2001}. If $X$ is a simplicial space, then $X$ is Segal if and only if $i_!X$ is a dendroidal Segal space; conversely, if $Y$ is dendroidal Segal, then $i^*Y$ is a simplicial Segal space.

\begin{rem}
    By \cite[Proposition~6.39]{Heuts2022}, the definition above is equivalent to requiring that the maps induced by inner horn inclusions $\horn{e}{T}\to T$ are weak equivalences. Using the tensor/enrichment adjunction and the previous remark, this can further be reformulated as stating that the map $X\to *$ lifts against every map of the form
    \begin{equation}\label{eq:pop_Segal}
        \pop{\cdot}{\horn{e}{T}}{T}{\partial\std{n}}{\std{n}}.
    \end{equation}
\end{rem}

By taking the left Bousfield localization of the Reedy model structure at the set of all spine inclusions $\spn{T}\to T$ we obtain the following (see \cite[Section~5]{Cisinski_Moerdijk_2}):

\begin{thm}
    There is a cofibrantly generated model structure on $\dsp$, denoted by $\dspsegal$, such that:
    \begin{itemize}
        \item[{\rm (i)}] The cofibrations are the normal monomorphisms.
        \item[{\rm (ii)}] The fibrant objects are the dendroidal Segal spaces.
        \item[{\rm (iii)}] The weak equivalences (resp.\ fibrations) between fibrant objects are the levelwise weak equivalences (resp.\ Reedy fibrations).
    \end{itemize}
\end{thm}

\begin{rem}\label{rm:inner_an_triv_cof_segal}
    More generally, a map $f:A\to B$ in $\dsp$ is a weak equivalence in $\dspsegal$ if and only if for every dendroidal Segal space $X$, the induced map
    \begin{equation*}
        \Map{B}{X}\to\Map{A}{X}
    \end{equation*}
    is a trivial fibration of simplicial sets. In particular, all inner horn inclusions $\horn{e}{T}\to T$, and more generally all inner anodyne maps of dendroidal sets, are trivial cofibrations in $\dspsegal$.
\end{rem}

Recall from \ref{sec:segal_preop} the functor
\begin{equation*}
    \gamma_*:\dsp\to\preop
\end{equation*}
right adjoint to the inclusion $\gamma^*$. It sends a dendroidal space to its largest subspace that is discrete at $\eta$. Since $\gamma^*$ is fully faithful, we will denote $\gamma^*\gamma_*X$ by simply $\gamma_*X$ and regard $\gamma_*X$ as a dendroidal space. The following observation will be useful:

\begin{prop}\label{pr:gamma_*X_Segal}
    Let $X$ be a dendroidal Segal space. Then $\gamma_*X$ is also a dendroidal Segal space.
\end{prop}

\begin{proof}
    We first show that $\gamma_*X$ is Reedy fibrant, that is, $\gamma_*X\to *$ has the right lifting property with respect to the maps
    \begin{equation*}
        \pop{\cdot}{\partial T}{T}{\horn{k}{n}}{\std{n}}
    \end{equation*}
    for evert tree $T$ and every $n$. If $T=\eta$, then $(\gamma_*X)_\eta$ is discrete and therefore lifts against every horn. If $T\neq\eta$, then $(\partial T)_\eta=T_\eta$, so the map above is an isomorphism at $\eta$. Any lift in $X$ (which exists because $X$ is Reedy fibrant) automatically lands inside $\gamma_*X$, giving the required lift:
    \begin{equation*}
        \begin{tikzcd}
        \popd{\cdot}{\partial T}{T}{\partial\std{n}}{\std{n}} \arrow[r] \arrow[d]                    & \gamma_*X \arrow[r, hook] & X \\
        T\cdot \std{n}. \arrow[rru] \arrow[ru, dashed] &              &  
        \end{tikzcd}
    \end{equation*}

    To show that $\gamma_*X$ satisfies the Segal condition, consider an inner horn inclusion
    \begin{equation*}
        \pop{\cdot}{\horn{e}{T}}{T}{\partial\std{n}}{\std{n}}
    \end{equation*}
    with $T$ having at least two vertices. As before, $(\horn{e}{T})_\eta=T_\eta$, so the same argument shows that lifts in $X$ already lie in $\gamma_*X$, and hence $\gamma_*X$ is Segal.
\end{proof}

\subsection{Dendroidal Segal spaces as models of homotopy operads}

Given a dendroidal Segal space, one can define colours, multi-mapping spaces, and the basic operadic constructions.

\begin{defn}
    Let $X$ be a dendroidal Segal space.
    \begin{enumerate}
        \item[{\rm (i)}] The \emph{space of colours} of $X$ is the space $X_\eta$. Its vertices will be called \emph{colours}.
        
        \item[{\rm (ii)}] For any sequence $x_1,\ldots,x_k,y\in X_{\eta,0}$ of colours, the \emph{multi-mapping space} $\text{map}_X(x_1,\ldots,x_k;y)$ is the (homotopy) fiber of the fibration
        \begin{equation*}
            X_{C_k}\longrightarrow X_{\partial C_k}\cong \Bigl(\prod_{1\leq i\leq k}X_\eta\Bigr)\times X_\eta
        \end{equation*}
        over the pair $((x_1,\ldots,x_k),y)$.

        \item[{\rm (iii)}] Given a tree with two vertices $T=C_l\circ_e C_k$ obtained by grafting two corollas, the Segal condition on $X$ implies that
        \begin{equation*}
            X_T=\Map{T}{X}\longrightarrow \Map{\horn{e}{T}}{X}\cong X_{C_l}\times_{X_\eta}X_{C_k}
        \end{equation*}
        is a trivial fibration. Choosing a section and composing with the inner face map of $T$ yields a composition map
        \begin{equation}\label{eq:comp_segal}
            \text{comp}_{l,k}:X_{C_l}\times_{X_\eta}X_{C_k}\longrightarrow X_T\longrightarrow X_{C_{k+l-1}},
        \end{equation}
        which is unique up to homotopy and also associative and unital up to homotopy. In particular, for colours $x_1,\ldots,x_k,y_1,\ldots,y_l,z\in X_{\eta,0}$ and $1\leq i\leq l$, we obtain an induced map on multi-mapping spaces (fibers)
        $$            \text{map}_X(y_1,\ldots,y_l;z)\times \text{map}_X(x_1,\ldots,x_k;y_i)\longrightarrow
            \text{map}_X(y_1,\ldots,y_{i-1},x_1,\ldots,x_k,y_{i+1},\ldots,y_l;z).
        $$

        \item[{\rm (iv)}] The \emph{homotopy operad} of $X$ is the operad $\ho X$ with set of colours $X_{\eta,0}$ and operations
        $$
            \ho X(x_1,\ldots,x_k;y)=\pi_0\text{map}_X(x_1,\ldots,x_k;y).
        $$
        Applying $\pi_0$ to the induced maps on multi-mapping spaces yields strict associativity and unitality (since $\pi_0$ preserves finite products), so $\ho X$ is indeed an operad.
    \end{enumerate}
\end{defn}

A convenient way to describe the operadic structure of $\ho X$ is obtained by considering the dendroidal space $\gamma_*X$, which is Segal by Proposition~\ref{pr:gamma_*X_Segal}. Note that $(\gamma_*X)_\eta$ is the discrete space $X_{\eta,0}$ and
$$
    \mapds{\gamma_*X}{x}{k}{y}=\mapds{X}{x}{k}{y}.
$$
Recall that the horizontal inclusion $i_h:\dset\to\dsp$ admits a left adjoint $(\pi_0)_*$ given by taking connected components
\begin{equation*}
    ((\pi_0)_*X)_T=\pi_0(X_T).
\end{equation*}
In particular, $(\pi_0)_*\gamma_*X$ takes the value $X_{\eta,0}$ at $\eta$. Since $\gamma_*X$ is Segal, for every tree $T$ the map
\begin{equation}\label{eq:pi_0_Map}
    \pi_0\Map{T}{\gamma_*X}\longrightarrow \pi_0\Map{\spn{T}}{\gamma_*X}
\end{equation}
is a bijection. We have that $\pi_0\Map{T}{\gamma_*X}=\dset(T,(\pi_0)_*\gamma_* X)$, and that $\Map{\spn{T}}{\gamma_*X}$ is an iterated pullback of spaces of the form $\Map{C_k}{\gamma_*X}$ over the discrete space $(\gamma_*X)_\eta$. Since $\pi_0$ preserves both coproducts and finite products, it also preserves pullbacks over discrete spaces. Thus, we can identify \eqref{eq:pi_0_Map} with the map
$$
    \dset(T,(\pi_0)_*\gamma_*X)\longrightarrow \dset(\spn{T},(\pi_0)_*\gamma_*X).
$$
Hence $(\pi_0)_*\gamma_*X$ is a dendroidal strict inner Kan complex, that is, a dendroidal set having unique fillers for spines of trees. By \cite[Proposition~6.4]{Heuts2022}, it is isomorphic to the dendroidal nerve of an operad $P$, whose partial compositions are encoded by
$$
    \dset(\spn{C_l\circ C_k},N_dP)\xrightarrow{\cong}\dset(C_l\circ C_k,N_dP)\xrightarrow{d_e}\dset(C_{k+l-1},N_dP),
$$
which arise from \eqref{eq:comp_segal} after applying $\pi_0\circ\gamma_*$. This yields the following result, which in fact could have been used to define $\ho X$:

\begin{prop}\label{pr:description_ho}
    For any dendroidal Segal space $X$, there is a natural isomorphism of dendroidal sets
   $$
        (\pi_0)_*\gamma_*X\cong N_d\ho X.
    $$
\end{prop}
Since the dendroidal nerve is fully faithful, we obtain:

\begin{cor}\label{co:description_ho}
    Let $X$ be a dendroidal Segal space. Then the homotopy operad ${\rm\ho} X$ is isomorphic to $\tau_d(\pi_0)_*\gamma_*X$.
\end{cor}

\begin{cor}\label{co:comut_j*_ho}
    Let $X$ be a dendroidal Segal space. Then $j^*\ho X\cong\ho(i^*X)$.
\end{cor}

\begin{proof}
    We have the following chain of isomorphisms
    \begin{equation*}       
    \begin{split}
           Nj^*\ho X &\cong i^*N_d\ho X \cong i^*((\pi_0)_*\gamma_*X)\\
           &\cong (\pi_0)_*i^*\gamma_*X
           \cong (\pi_0)_*\gamma_*i^*X \cong N\ho(i^*X).
      \end{split}
      \end{equation*}
    Since $N$ is fully faithful, the claim follows.
\end{proof}

\begin{rem}
    All the operadic notions introduced above for a dendroidal Segal space depend only on the Segal space $\gamma_*X$. In particular, $X$ and $\gamma_*X$ have the same set of objects, the same multi-mapping spaces, and the same homotopy category. It is therefore natural to regard both as models of the same $\infty$-operad, with $\gamma_*X\to X$ an equivalence of $\infty$-operads. However, this inclusion is not in general a weak equivalence in $\dspsegal$. This suggests a modification on the model structure so that such maps become weak equivalences.
\end{rem}

\subsection{Dwyer--Kan equivalences}

The following definition arises naturally from the parallels established above between dendroidal Segal spaces and operads.

\begin{defn}\label{df:DKeq}
    Let $f\colon X\to Y$ be a map of dendroidal Segal spaces. We say that $f$ is a \emph{Dwyer--Kan equivalence} if the following conditions hold:
    \begin{enumerate}
        \item[{\rm (i)}] For every sequence $x_1,\ldots,x_k,y\in X_{\eta,0}$, the induced map
        $$            \text{map}_X(x_1,\ldots,x_k;y)\longrightarrow \text{map}_Y\bigl(f(x_1),\ldots,f(x_k);f(y)\bigr)
        $$
        is a weak homotopy equivalence.
        \item[{\rm (ii)}] The induced map of operads $\ho X\to \ho Y$ is essentially surjective.
    \end{enumerate}
\end{defn}

A basic example is the following:

\begin{ex}\label{ex:gammaXtoX_DK}
    The inclusion $\gamma_*X\hookrightarrow X$
    induces isomorphisms on all multi-mapping spaces and is surjective on vertices at $\eta$. Hence it is a Dwyer--Kan equivalence.
\end{ex}

Any map satisfying condition (i) of Definition~\ref{df:DKeq} is called \emph{homotopically fully faithful}. We now record a useful characterization.

\begin{prop}\label{pr:char_fullyfaithful}
    For a map $f:X\to Y$ of dendroidal Segal spaces, the following are equivalent:
    \begin{enumerate}
        \item[{\rm (i)}]\label{it:1} The map $f$ is homotopically fully faithful.
        \item[{\rm (ii)}]\label{it:2} For each corolla $C_k$, the map
       $$
            X_{C_k}\longrightarrow Y_{C_k}\times_{Y_{\partial C_k}}X_{\partial C_k}
        $$    is a weak equivalence in $\sset$.
        \item[{\rm (iii)}]\label{it:3} For every tree $T$, the map
        $$
            X_T\longrightarrow Y_T\times_{(\text{cosk}_\eta Y_\eta)_T}(\text{cosk}_\eta X_\eta)_T
        $$
        is a weak equivalence in $\sset$;
        \item[{\rm(iv)}]\label{it:4} For every tree $T$ with at least one vertex, the map
        $$
            X_T\to Y_T\times_{Y_{\partial T}}X_{\partial T}
        $$
        is a weak equivalence in $\sset$.
    \end{enumerate}
\end{prop}

\begin{proof}
    The equivalence between {\rm (i)} and ${\rm (ii)}$ follows from the commutative diagram
    \begin{equation}\label{dg:functoriality}
        \begin{tikzcd}
            X_{C_k} \arrow[d, two heads] \arrow[r] & Y_{C_k} \arrow[d, two heads] \\
            X_{\partial C_k} \arrow[r] & Y_{\partial C_k}
        \end{tikzcd}
    \end{equation}
    Since vertical arrows are fibrations, the square is a homotopy pullback if and only if the induced maps on all fibers are weak equivalences, and also if and only if the comparison map from $X_{C_k}$ to the strict pullback is a weak equivalence.

    Conditions (iii) and (iv) both imply (ii), since the case $T=C_k$ recovers precisely condition (ii). For the converse implications, we proceed by induction on the number of vertices of $T$. For $T=C_k$, the statement is exactly (ii) and (iii) is automatically true for $T=\eta$. Thus we may assume $T$ has at least two vertices. Consider the diagram
    \begin{equation}\label{dg:fullyfaithful}
        \begin{tikzcd}
            X_T \arrow[d] \arrow[r, two heads] & \Map{\partial T}{X} \arrow[d] \arrow[r, two heads] & \Map{\spn{T}}{X} \arrow[d] \arrow[r, two heads] & \Map{\text{sk}_\eta T_\eta}{X} \arrow[d] \\
            Y_T \arrow[r, two heads] & \Map{\partial T}{Y} \arrow[r, two heads] & \Map{\spn{T}}{Y} \arrow[r, two heads] & \Map{\text{sk}_\eta T_\eta}{Y}
        \end{tikzcd}
    \end{equation}
    All horizontal maps are Kan fibrations between Kan complexes because $X$ and $Y$ are Reedy fibrant. Condition (iii) (resp.\ (iv)) is equivalent to the total rectangle (resp.\ left square) being a homotopy pullback. The composites of the left and middle horizontal arrows are weak equivalences by the Segal condition; hence the composite of the left and middle squares is a homotopy pullback. Moreover, the maps
    $$        \text{sk}_\eta T_\eta\to \spn{T}\to \partial T
    $$
    are formed by a sequence of pushouts of maps of the form $\partial S\to S$ for trees $S$ with strictly fewer vertices than $T$ (via the skeletal filtration of normal dendroidal sets). By the induction hypothesis, the middle and right squares are homotopy pullbacks. By the pasting/cancellation laws for homotopy pullbacks, the claim follows.
\end{proof}

\begin{rem}\label{rm:alt_fullyfaithful}
    For any map $f:X\to Y$ of dendroidal Segal spaces, the object
    $$
        Y\times_{\text{cosk}_\eta Y_\eta} \text{cosk}_\eta X_\eta
    $$
    is again a dendroidal Segal space. This follows by inspecting the matching maps and using that $\spn{T}_\eta=T_\eta$. Proposition~\ref{pr:char_fullyfaithful} then shows that $f$ is homotopically fully faithful if and only if
    \begin{equation*}
        X\to Y\times_{\text{cosk}_\eta Y_\eta} \text{cosk}_\eta X_\eta
    \end{equation*}
    is a Reedy weak equivalence. But this is the case if and only if $f$ is a weak equivalence in $\dspsegal$.
\end{rem}

\begin{rem}\label{rm:fullyfaithful_reedy_fib_case}
    The argument of Proposition~\ref{pr:char_fullyfaithful} adapts to show that for any Reedy fibration $f:X\to Y$ between arbitrary dendroidal spaces (not necessarily Segal), conditions (iii) and (iv) are equivalent. Indeed, in diagram~\eqref{dg:fullyfaithful} but ignoring the spines, the vertical maps are fibrations. Since $\sset$ is right proper, conditions (iii) and (iv) correspond to the total rectangle and the left square being homotopy pullbacks, respectively. If condition (iii) holds, then by induction the right square is a homotopy pullback, so that the left square is a pullback. Assuming condition (iv) the total rectangle is a pullback again by the same inductive argument based on the skeletal filtration of $\text{sk}_\eta T_\eta\to T$. 
\end{rem}

\begin{ex}
    Any weak equivalence $f:X\to Y$ in $\dspsegal$ between dendroidal Segal spaces is a Dwyer--Kan equivalence. Indeed, $f$ is a Reedy weak equivalence, hence $f_\eta$ and $f_{C_k}$ are weak equivalences in $\sset$, making \eqref{dg:functoriality} a homotopy pullback for every corolla and ensuring that $\pi_0 f_\eta$ is surjective. The former gives homotopical fully faithfulness, and the latter essential surjectivity.
\end{ex}

\subsection{Isofibrations}

We now introduce the class of morphisms that will serve as the fibrations between fibrant objects in the model structure developed later.

\begin{defn}
    A map $f:X\to Y$ of dendroidal Segal spaces is an \emph{isofibration} if it is a Reedy fibration and the induced map
    $$
        j^*\ho X\to j^*\ho Y
    $$
    is an isofibration of categories.
\end{defn}

\begin{rem}\label{rm:isofib_i^*}
    By Corollary~\ref{co:comut_j*_ho} we have a natural isomorphism
    $$
        j^*\ho X\cong \ho(i^*X)
    $$
    for every dendroidal Segal space $X$. Thus a Reedy fibration $f:X\to Y$ between dendroidal Segal spaces is an isofibration if and only if the induced map of simplicial Segal spaces    $i^*f:i^*X\to i^*Y$ is an isofibration in the sense of \cite[Definition~2.13]{moser2024}.
\end{rem}

This observation allows one to transport many basic properties of isofibrations of simplicial Segal spaces directly to the dendroidal setting. The next two results illustrate this principle.

\begin{prop}\label{pr:char_isofib}
    A Reedy fibration $f:X\to Y$ between dendroidal Segal spaces is an isofibration if and only if it has the right lifting property with respect to either inclusion
    \begin{equation*}
        q:\Omega[\eta]\cong i_!N[0]\to i_!NI[1].
    \end{equation*}
\end{prop}
\begin{proof}
    By Remark~\ref{rm:isofib_i^*}, the statement reduces to the corresponding one for simplicial Segal spaces. The latter asserts that a morphism is an isofibration precisely when it lifts against $N[0]\to NI[1]$; see \cite[Proposition~2.15]{moser2024}.
\end{proof}

\begin{prop}\label{pr:exp_isofib}
    Let $p:X\to Y$ be an isofibration between dendroidal Segal spaces. Then the induced maps
    $$
        \expd{X}{i_!NI[1]}\to \expd{Y}{i_!NI[1]}\times_{Y^2} X^2
    $$
    and
    \begin{equation*}
        \expd{X}{i_!NI[1]}\to \expd{Y}{i_!NI[1]}\times_{Y} X
    \end{equation*}
    are isofibrations.
\end{prop}

\begin{proof}
    This follows immediately from the analogous result for simplicial Segal spaces~\cite[Lemma~2.18]{moser2024}, together with Remark~\ref{rm:isofib_i^*}, the fact that $i^*$ preserves limits and Segal spaces, and the natural isomorphism $i^*(\expd{Y}{i_!X})\cong (i^*Y)^{X}$
    for $X$ a simplicial space and $Y$ a dendroidal space.
\end{proof}

\subsection{Compatibility with bifunctors}
In this section we establish several compatibility properties of $\dsp_{\text{Segal}}$ with both the simplicial tensor $\cdot$ and the tensor product $\otimes$ of dendroidal spaces. To streamline notation, we denote the pushout product of two maps $f\colon A\to B$ and $g\colon C\to D$ with respect to the simplicial cotensor $\cdot$ by
$$
    f\boxdot g\colon \pop{\cdot}{A}{B}{C}{D},
$$
the pushout product associated to the tensor product $\otimes$ by
$$
    f\boxtimes g\colon \pop{\otimes}{A}{B}{C}{D},
$$
and the usual cartesian pushout product by
$$
    f\boxc g\colon \pop{\times}{A}{B}{C}{D},
$$
which will only be used for maps of simplicial sets. We begin with a straightforward observation.

\begin{lem}\label{lm:prod}
    Let $f$ be a normal monomorphism of dendroidal \emph{sets} (i.e.\ discrete dendroidal spaces) and let $g$ be a cofibration of simplicial sets. Then $f\boxdot g$ is a Reedy cofibration (that is, a normal monomorphism), and it is trivial whenever $g$ is.
\end{lem}

\begin{proof}
    It suffices to treat the case where $f$ is a boundary inclusion of a tree and $g$ is a boundary or horn inclusion of a simplex. In this case the statement follows immediately from the description of the generating Reedy (trivial) cofibrations~\eqref{eq:gen_cof_Reedy} and~\eqref{eq:gen_triv_cof_Reedy}.
\end{proof}

We now prove that both the Reedy and Segal model structures on $\dsp$ are simplicial model structures.

\begin{prop}\label{pr:Reedy_simplicial}
    The simplicial tensor
    $$
        (-)\cdot(-)\colon \dsp\times\sset\to\dsp
    $$
    is a left Quillen bifunctor with respect to the Reedy model structure on $\dsp$ and the Kan--Quillen model structure on $\sset$.
\end{prop}

\begin{proof}
    Let $f=f_1\boxdot f_2$ be a generating Reedy cofibration of dendroidal sets, where $f_1$ and $f_2$ are boundary inclusions $\partial T\to T$ and $\partial\std{n}\to\std{n}$, and let $g\colon A\to B$ be any monomorphism of simplicial sets. Then
    \begin{equation}\label{eq:ass1}
        f\boxdot g=(f_1\boxdot f_2)\boxdot g=f_1\boxdot (f_2\boxc g).
    \end{equation}
    By Lemma~\ref{lm:prod}, $f\boxdot g$ is a Reedy cofibration.

    If $g$ is a trivial cofibration, then $f_2\boxc g$ is a trivial cofibration of simplicial sets, hence the map~$f\boxdot g$ is a Reedy trivial cofibration by Lemma~\ref{lm:prod}. If instead $f_2$ is a horn inclusion (making $f$ a generating Reedy trivial cofibration), then $f_2\boxc g$ is again a trivial cofibration in $\sset$, so $f\boxdot g$ is a Reedy trivial cofibration.
\end{proof}

\begin{prop}\label{pr:Segal_simplicial}
    The simplicial tensor
    \begin{equation*}
        (-)\cdot(-)\colon \dsp\times\sset\to\dsp
    \end{equation*}
    is a left Quillen bifunctor with respect to the Segal model structure on $\dsp$ and the Kan--Quillen model structure on $\sset$.
\end{prop}

\begin{proof}
    By Proposition~\ref{pr:Reedy_simplicial}, it remains to show the following: if $f\colon A\to B$ is a Reedy cofibration which is a weak equivalence in $\dsp_{\text{Segal}}$, and $g\colon C\to D$ is a cofibration of simplicial sets, then $\varphi=\Map{f\boxdot g}{X}$
    is a weak equivalence of simplicial sets whenever $X$ is Segal. Since $X$ is Reedy fibrant and $f\boxdot g$ is a Reedy cofibration, $\varphi$ is a Kan fibration; thus it suffices to show that it is trivial.

    A simplex boundary inclusion $h$ lifts on the left against $\varphi$ if and only if $g\boxc h$ lifts on the left against
    \begin{equation*}
        \Map{f}{X}\colon \Map{B}{X}\to\Map{A}{X}.
    \end{equation*}
    This lifting holds because $f$ is a weak equivalence in $\dsp_{\text{Segal}}$ and $X$ is Segal. Hence $\varphi$ is a trivial fibration, completing the proof.
\end{proof}

We now consider the tensor product in the case where one of the factors is simplicial.

\begin{prop}\label{pr:Reedy_tensor}
    The restricted tensor product
    $$
        (-)\otimes i_!(-)\colon \dsp\times\ssp\to\dsp
    $$
    is a left Quillen bifunctor with respect to the Reedy model structures.
\end{prop}

\begin{proof}
    Let $f=f_1\boxdot f_2$ and $g=g_1\boxdot g_2$ be generating cofibrations in the Reedy model structures on dendroidal and simplicial sets, respectively. Since both $\otimes$ and $i_!$ preserve the simplicial tensors appearing in these expressions, we obtain
    \begin{equation}\label{eq:calculo_boxtimes}
        f\boxtimes i_!g
        = (f_1\boxdot f_2)\boxtimes i_!(g_1\boxdot g_2)
        = (f_1\boxtimes i_!g_1)\boxdot (f_2\boxc g_2).
    \end{equation}
    By \cite[Proposition~4.21]{Heuts2022}, the map $f_1\boxtimes i_!g_1$ is a normal monomorphism of dendroidal \emph{sets}, while $f_2\boxc g_2$ is a cofibration of simplicial sets. Lemma~\ref{lm:prod} therefore implies that $f\boxtimes i_!g$ is a Reedy cofibration.

    If either $f_2$ or $g_2$ is a horn inclusion (so that $f$ or $g$ is a generating \emph{trivial} cofibration), then $f_2\boxc g_2$ is a trivial cofibration in $\sset$, and Lemma~\ref{lm:prod} again implies that $f\boxtimes i_!g$ is a trivial Reedy cofibration.
\end{proof}

\begin{prop}\label{pr:Segal_tensor}
    The restricted tensor product
    \begin{equation*}
        (-)\otimes i_!(-)\colon \dsp\times\ssp\to\dsp
    \end{equation*}
    is a left Quillen bifunctor with respect to the Segal model structures.
\end{prop}

\begin{proof}
    Using Proposition~\ref{pr:Reedy_tensor} together with the Quillen adjunction between the Reedy and Segal model structures on $\dsp$, we obtain a left Quillen functor
    $$
        (-)\otimes i_!(-)\colon \dsp_{\mathrm{Reedy}}\times \ssp_{\mathrm{Reedy}}\to\dsp_{\mathrm{Segal}}.
    $$
    By \cite[Proposition~3.21]{Gutierrez2017}, it remains to verify the following two conditions:
    \begin{itemize}
        \item [{\rm (i)}] For every dendroidal Segal space $X$ and simplicial space $A$, the dendroidal space $\expd{X}{i_!A}$ is Segal.
        \item[{\rm (ii)}] For every dendroidal Segal space $X$ and normal dendroidal space $A$, the simplicial space $i^*(\expd{X}{A})$ is Segal.
    \end{itemize}
    We prove only condition (i); the second is similar. Since $\expd{X}{i_!A}$ is Reedy fibrant by Proposition~\ref{pr:Reedy_tensor}, it is a dendroidal Segal space if and only if
    $$
        \Map{T}{\expd{X}{i_!A}} \to \Map{\horn{e}{T}}{\expd{X}{i_!A}}
    $$
    is a trivial fibration for every tree $T$ and inner edge $e$. This is equivalent to the lifting property against all maps of the form $f=f_1\boxdot f_2$, where $f_1\colon \spn{T}\to T$ is an inner horn inclusion and $f_2\colon \partial\std{n}\to\std{n}$ is a boundary inclusion of simplices. By the adjunction~\eqref{eq:adj_otimes}, this is in turn equivalent to $X$ lifting against
    \begin{equation}\label{eq:popd}
        f\boxtimes i_!(\emptyset\to A).
    \end{equation}

    We now prove a stronger statement: for \emph{any} Reedy cofibration $g$ in $\ssp$, the map $f\boxtimes i_!g$ lifts against $X$. Let $g=g_1\boxdot g_2$ be a generating Reedy cofibration in $\ssp$. Using again~\eqref{eq:calculo_boxtimes}, we have
    $$
        f\boxtimes i_!g = (f_1\boxtimes i_!g_1)\boxdot (f_2\boxc g_2).
    $$
    The map $f_1\boxtimes i_!g_1$ is inner anodyne by \cite[Corollary~6.26]{Heuts2022}, while $f_2\boxc g_2$ is a cofibration of simplicial sets. Thus Proposition~\ref{pr:Segal_simplicial} implies that $f\boxtimes i_!g$ is a trivial cofibration in the Segal model structure, and therefore lifts against $X$, completing the proof.
\end{proof}

We now show that, although $\dsp$ is not strictly enriched, tensored, or cotensored over $\ssp$, and hence $\dsp_{\mathrm{Segal}}$ is not an enriched $\ssp_{\mathrm{Segal}}$-model category, its weak enrichment nevertheless induces a homotopical enrichment over $\ssp_{\mathrm{Segal}}$:

\begin{thm}\label{th:Segal_ho_enriched}
    The tensor product of dendroidal spaces equips $\dsp_{\mathrm{Segal}}$ with the structure of a homotopically $\ssp_{\mathrm{Segal}}$-enriched model category in the sense of \cite[Definition~3.5.4]{Heuts2015}. In particular, the homotopy category of $\dsp_{\mathrm{Segal}}$ is enriched over the homotopy category of $\ssp_{\mathrm{Segal}}$.
\end{thm}

\begin{proof}
    Two conditions must be verified. The first is precisely the content of Proposition~\ref{pr:Segal_tensor}. It remains to show that for any monomorphisms $f,g$ in $\ssp_{\mathrm{Segal}}$ and any normal monomorphism $h$ in $\dsp_{\mathrm{Segal}}$, the map labelled $\phi_{f,g,h}$ in the diagram
    $$
        \begin{tikzcd}
	    \bullet & \bullet \\
	    \bullet & \bullet \\
	    && \bullet
	    \arrow["\alpha", from=1-1, to=1-2]
	    \arrow["{i_!(f\boxc g)\boxtimes h}"', from=1-1, to=2-1]
	    \arrow[from=1-2, to=2-2]
	    \arrow["{i_!f\boxtimes(i_!g\boxtimes h)}", bend left, from=1-2, to=3-3]
	    \arrow[from=2-1, to=2-2]
	    \arrow["\alpha"', bend right, from=2-1, to=3-3]
	    \arrow["{\phi_{f,g,h}}"{description}, from=2-2, to=3-3]
        \end{tikzcd}
    $$
    is a trivial cofibration in $\dsp_{\mathrm{Segal}}$. This diagram arises from the naturality of the associator map $\alpha$ appearing in~\eqref{eq:weak_enrichment}.
    Using the same reductions as in the previous proofs, we may assume that $f$, $g$, and $h$ are boundary inclusions of representables, so that we have for example
    $$
        h=h_1\boxdot h_2
    $$
    with $h_1\colon\partial T\to T$ and $h_2\colon\partial\std{n}\to\std{n}$. Since $\boxtimes$ preserves $\boxdot$, we may rearrange the corresponding factors in the diagram above to obtain an expression of the form
    $$
    \phi_{f,g,h}=\phi_{f_1,g_1,h_1}\boxdot (f_2\boxc g_2\boxc h_2),
    $$
    where $\phi_{f_1,g_1,h_1}$ is a map of dendroidal \emph{sets}. By \cite[Lemma~3.8.1]{Heuts2015}, the map $\phi_{f_1,g_1,h_1}$ is inner anodyne. Proposition~\ref{pr:Segal_simplicial} then implies that $\phi_{f,g,h}$ is a trivial cofibration in $\dsp_{\mathrm{Segal}}$, completing the proof.
\end{proof}

\section{Constructing the model structure}

We now focus on proving the existence of the model structure of Theorem~\ref{thm:main}, whose statement we reproduce here for ease of reference:

\begin{thm}\label{thm:main_existence}
    There exists a left proper, cofibrantly generated model structure on $\dsp$, which we denote by $\dsp_{\mathrm{Op}}$, such that:
    \begin{itemize}
        \item[{\rm (i)}] The cofibrations are the normal monomorphisms $f\colon X\to Y$ for which there exists a set $R$ and a weak equivalence $X_\eta \coprod R \to Y_\eta$ extending $f_\eta$.
        \item[{\rm (ii)}] the fibrant objects are the dendroidal Segal spaces.
        \item[{\rm (iii)}] The weak equivalences between fibrant objects are the Dwyer–Kan equivalences;
        \item[{\rm (iv)}] The fibrations between fibrant objects are the isofibrations.
    \end{itemize}
    Moreover, this model structure is homotopically enriched (in the sense of \cite{Heuts2015}) over the structure $\sspcat$ introduced in \cite{moser2024}, which in fact arises as the slice of $\dsp_{\mathrm{Op}}$ over $\eta$.
\end{thm}

To construct this model structure we appeal to the fibrantly generated recognition theorem of \cite{guetta2023}, which we recall for the reader’s convenience:

\begin{thm}\label{th:fibrantly-induced}
    Let $\mathcal{C}$ be a locally presentable category, and let $\mathcal{I}$ and $\mathcal{J}$ be sets of morphisms in $\mathcal{C}$ such that $\mathcal{J} \subset \mathcal{I}\text{-cof}$. Suppose moreover that we are given a class $\mathcal{W}_f$ of morphisms between $\mathcal{J}$-fibrant objects satisfying:
    \begin{enumerate}[label=\textnormal{(\Roman*)}]
        \item \label{condI} $\mathcal{W}_f$ satisfies 2-out-of-6.
        \item \label{condII} There exists a class $\overline{W}$ of morphisms in $\mathcal{C}$ whose restriction to the morphisms between $\mathcal{J}$-fibrant objects is precisely $\mathcal{W}_f$, and for which $\overline{W}$, viewed as a full subcategory of the category of arrows $\mathcal{C}^{[1]}$, is accessible.
        \item \label{condIII} For every $\mathcal{J}$-fibrant object $X$, the diagonal morphism factors as
        $$
            X \xrightarrow{w} \mathrm{Path}(X) \xrightarrow{p} X\times X
        $$
        with $w \in \mathcal{W}_f$ and $p \in \mathcal{J}\text{-fib}$;
        \item \label{condIV} $\mathcal{J}\text{-fib} \cap \mathcal{W}_f \subset \mathcal{I}\text{-fib}$;
        \item \label{condV} $\mathcal{I}\text{-fib} \subset \mathcal{W}$, where $\mathcal{W}$ denotes the class of morphisms of $\mathcal{C}$ admitting a $\mathcal{J}$-fibrant replacement lying in $\mathcal{W}_f$.
    \end{enumerate}
    Then $\mathcal{C}$ admits a cofibrantly generated model structure in which the cofibrations are the $\mathcal{I}$-cofibrations, the fibrant objects are the $\mathcal{J}$-fibrant objects, and the weak equivalences are the morphisms in $\mathcal{W}$. Moreover, among fibrant objects the weak equivalences and fibrations are precisely $\mathcal{W}_f$ and $\mathcal{J}\text{-fib}$, respectively.
\end{thm}

Our verification of the hypotheses of Theorem~\ref{th:fibrantly-induced} follows closely the strategy of \cite{moser2024}. We begin by analysing the class of cofibrations that will appear in the resulting model structure.

\subsection{Generating cofibrations}

\begin{defn}
    We define $\mathcal{I}$ to be the set of maps in $\dsp$ consisting of:
    \begin{enumerate}
        \item[{\rm (i)}] For every tree $T$ with at least one vertex and every $n\geq 0$, the map
        \[
            \pop{\cdot}{\partial T}{T}{\partial\std{n}}{\std{n}}.
        \]
        \item[{\rm (ii)}] For all $n\geq 1$ and $0\leq k\leq n$, the map
        \[
            \eta\cdot \horn{k}{n}\longrightarrow \eta\cdot \std{n}.
        \]
        \item[{\rm (iii)}] The map
        \[
            \emptyset\longrightarrow \eta\cdot \std{0}.
        \]
    \end{enumerate}
\end{defn}

\begin{prop}\label{pr:char_I-fibs}
    A map $p\colon X\to Y$ is an $\mathcal{I}$-fibration if and only if:
    \begin{enumerate}
        \item[{\rm (i)}] The map $p_\eta\colon X_\eta\to Y_\eta$ is a surjective Kan fibration.
        \item[{\rm (ii)}] For every tree $T\neq\eta$, the induced map
        \[
            p_T\colon X_T\longrightarrow Y_T\times_{Y_{\partial T}} X_{\partial T}
        \]
        is a trivial Kan fibration.
    \end{enumerate}
\end{prop}

\begin{proof}
    This is a direct unpacking of the lifting properties with respect to the generators of $\mathcal{I}$.  
    The second type of generating cofibrations ensures that $p_\eta$ is a Kan fibration, and the third forces $p_\eta$ to be surjective on vertices, and hence surjective. The generators of the first type then impose the stated condition (ii) for all $T\neq\eta$.
\end{proof}

We next identify the corresponding class of cofibrations. For this we first define the following notion.

\begin{defn}
    Let $f\colon X\to Y$ be a monomorphism of simplicial sets. We say that $f$ is an \emph{extension by a set} if there exists a set $R$ and a weak equivalence
    \begin{equation}\label{eq:ext}
        X\coprod R \longrightarrow Y
    \end{equation}
    extending the morphism~$f$.
\end{defn}

Such maps are not neccessarily literal extensions by sets, but they behave this way homotopically.  
Since the map in \eqref{eq:ext} induces an injection on $\pi_0$ and is a monomorphism on each component of the domain, then it is necessarily a trivial cofibration. Hence extensions by a set may be viewed as cofibrations that are “almost” trivial, differing only by the addition of a set.

\begin{rem}\label{rm:char_ext}
    If $f\colon X\to Y$ is an extension by a set and $Y$ is connected, then either $f$ is a trivial cofibration or else $X=\emptyset$ and $Y$ is contractible.  
    Conversely, any map of these two types is an extension by a set.  
    Moreover, $f\coprod g$ is an extension by a set if and only if both $f$ and $g$ are.  
    It follows that a map is an extension by a set if and only if it is of one of these two types over every connected component of its target.
\end{rem}

The class of such maps can be described as follows.

\begin{prop}\label{pr:ext_by_set_char}
    The class of extensions by sets is the smallest saturated class of maps of simplicial sets containing all horn inclusions $\horn{k}{n}\to \std{n}$ and the map $\emptyset \to \std{0}$.
\end{prop}

\begin{proof}
    The fact that the class of extensions by sets is saturated follows from Remark \ref{pr:ext_by_set_char} and the fact that retracts, pushouts and transfinite compositions commute with disjoint unions. Indeed, a retract of any of the two types of maps in Remark \ref{pr:ext_by_set_char} gives again a map of the same type, a transfinite composition of maps of the two types is again a map of one of the two types, and a pushout of any of those two types is either a trivial cofibration or a disjoint union with a contractible space, which is an extension by a set. Moreover, both $\emptyset\to\std{0}$ and all horn inclusions are extensions by sets, so the class generated by the former is contained in the latter.
    Finally, if $f\colon A\to B$ is an extension by a set, then by definition there exists a set $R$ and a factorization of $f$ as
    $$
        A \longrightarrow A\coprod R \longrightarrow B,
    $$
    where the first map lies in the saturated class generated by $\emptyset\to\std{0}$ and the second lies in the one generated by horn inclusions. Thus $f$ is contained in the class generated by both.
\end{proof}

We can now characterize the $\mathcal{I}$-cofibrations.

\begin{prop}\label{pr:char_I-cof}
    A map $f\colon X\to Y$ in $\dsp$ is an $\mathcal{I}$-cofibration if and only if:
    \begin{enumerate}
        \item[{\rm (i)}] The map $f$ is a normal monomorphism;
        \item[{\rm (ii)}] The map $f_\eta\colon X_\eta\to Y_\eta$ is an extension by a set.
    \end{enumerate}
\end{prop}

\begin{proof}
    Both the class of normal monomorphisms and the class of maps whose $\eta$-component is an extension by a set are saturated, and their intersection contains the generating maps in $\mathcal{I}$.  
    Hence every $\mathcal{I}$-cofibration satisfies the two stated conditions.

    Conversely, suppose $f$ is a normal monomorphism with $f_\eta$ an extension by a set.  
    Using the relative skeletal filtration, we may factor $f$ as
    $$
        X \xrightarrow{f_1} X\coprod_{\mathrm{sk}_\eta X_\eta}\mathrm{sk}_\eta Y_\eta \xrightarrow{f_2} Y
    $$
    where $f_1$ attaches only cells of the form $\eta\cdot\std{n}$, and $f_2$ attaches cells of the form $T\cdot \std{n}$ with $T$ having at least one vertex.  
    Thus $f_2$ is a (possibly transfinite) composition of pushouts of maps of type~(i) in $\mathcal{I}$.  
    Since $(f_1)_{\eta}=f_\eta$ is an extension by a set, Proposition~\ref{pr:ext_by_set_char} implies that $f_1$ is a composition of pushouts of maps of types~(ii) and~(iii) in $\mathcal{I}$.  
    Therefore $f$ is an $\mathcal{I}$-cofibration.
\end{proof}

\subsection{Generating anodyne extensions}

We now introduce the class of trivial cofibrations that will play a central role in the construction of the model structure.

\begin{defn}
    Let $\mathcal{J}$ be the set of maps in $\dsp$ consisting of:
    \begin{enumerate}
        \item[{\rm (i)}] For all $m\geq 0$, $n\geq 1$, and $0\leq k\leq n$, the map
        $$
            \pop{\cdot}{\partial \om{T}}{\om{T}}{\horn{k}{n}}{\std{n}}.
        $$
        \item[{\rm (ii)}] For all $m\geq 2$ and $n\geq 0$, the map
        $$
            \pop{\cdot}{\spn{T}}{\om{T}}{\partial\std{n}}{\std{n}}.
        $$
        \item[{\rm (iii)}] Either endpoint inclusion
        $$
            \om{\eta}\cong i_!N[0]\longrightarrow i_!NI[1].
        $$
    \end{enumerate}
    We denote by $\mathcal{J}_0\subset \mathcal{J}$ the subset consisting of the maps of types~(i) and~(ii).
\end{defn}

Every$\mathcal{J}_0$-cofibration is, in particular, a trivial cofibration in the Segal model structure $\dsp_{\mathrm{Segal}}$.  
A dendroidal space is $\mathcal{J}_0$-fibrant if and only if it is a Segal space, and a map between Segal spaces is a $\mathcal{J}_0$-fibration if and only if it is a Reedy fibration. Furthermore, Proposition~\ref{pr:char_I-cof} shows that $\mathcal{J}\subset \mathcal{I}\text{-cof}$, so the hypothesis $\mathcal{J}\subset\mathcal{I}$-cof in Theorem~\ref{th:fibrantly-induced} is satisfied.

\begin{prop}\label{pr:J_0_replacement}
    Let $X$ be a dendroidal space. Then:
    \begin{enumerate}
        \item[{\rm (i)}] $X$ is $\mathcal{J}_0$-fibrant if and only if it is $\mathcal{J}$-fibrant.
        \item[{\rm(ii)}] There exists a $\mathcal{J}$-fibrant replacement $r_X\colon X\to\widetilde{X}$ such that $r_X$ is a $\mathcal{J}_0$-cofibration.
    \end{enumerate}
\end{prop}

\begin{proof}
    For (i), one implication is immediate; for the other, note that the map $i_!N[0]\to i_!NI[1]$ admits a retract, hence lifts on the left against every map $X\to *$, and thus imposes no additional fibrancy condition beyond those in $\mathcal{J}_0$.

    For (2), apply the small object argument to the map $X\to *$ using the generating set $\mathcal{J}_0$.  
    This yields a factorization
    $$
        X\xrightarrow{r_X} \widetilde{X}\longrightarrow *
    $$
    in which $r_X$ is a $\mathcal{J}_0$-cofibration and $\widetilde{X}$ is $\mathcal{J}_0$-fibrant, hence $\mathcal{J}$-fibrant by (i).
\end{proof}

\begin{prop}\label{pr:J0_trivial_at_eta}
    Let $f\colon X\to Y$ be a $\mathcal{J}_0$-cofibration. Then the induced map
    $$
        f_\eta\colon X_\eta\longrightarrow Y_\eta
    $$
    is a trivial cofibration of simplicial sets.
\end{prop}

\begin{proof}
    The class of maps in $\dsp$ whose $\eta$-component is a trivial cofibration is saturated and contains all generators in $\mathcal{J}_0$.
\end{proof}

\begin{prop}\label{pr:J-fib_iff_isofib}
    A dendroidal space $X$ is $\mathcal{J}$-fibrant if and only if it is Segal.  
    If $f\colon X\to Y$ is a map between dendroidal Segal spaces, then $f$ is a $\mathcal{J}$-fibration if and only if it is an isofibration.
\end{prop}

\begin{proof}
    The first statement follows immediately from Proposition~\ref{pr:J_0_replacement}.  
    For the second, $f$ is a Reedy fibration precisely when it lifts against $\mathcal{J}_0$, and by Proposition~\ref{pr:char_isofib} it is an isofibration precisely when, in addition, it lifts against
    $$
        i_!N[0]\longrightarrow i_!NI[1].
    $$
    This is exactly the additional generator in $\mathcal{J}$, completing the characterization.
\end{proof}

\subsection{Weak equivalences}

We define $\mathcal{W}_f$ to be the class of Dwyer--Kan equivalences between dendroidal Segal spaces.  
Since equivalences of operads and weak equivalences of simplicial sets both satisfy the 2-out-of-6 property, we obtain:

\begin{prop}
    The class $\mathcal{W}_f$ satisfies the 2-out-of-6 property. $\hfill\qed$
\end{prop}

Thus condition~\ref{condI} of Theorem~\ref{th:fibrantly-induced} holds.  
We next verify condition~\ref{condII}:

\begin{prop}
    There exists a class $\overline{\mathcal{W}}$ of morphisms in $\dsp$ such that:
    \begin{itemize}
        \item[{\rm (i)}] The restriction of $\overline{\mathcal{W}}$ to the morphisms between dendroidal Segal spaces is exactly $\mathcal{W}_f$.
        \item[{\rm (ii)}] Regarded as a full subcategory of $(\dsp)^{[1]}$, the class $\overline{\mathcal{W}}$ is accessible.
    \end{itemize}
\end{prop}

\begin{proof}
    Consider the functor
    $$
        \Phi\colon (\dsp)^{[1]}\longrightarrow (\op)^{[1]}\times \prod_{k\in\mathbb{N}} (\sset)^{[1]}
    $$
    sending a map $X\to Y$ to the tuple consisting of
    $$
        \tau_d(\pi_0)_*\gamma_* X\;\longrightarrow\; \tau_d(\pi_0)_*\gamma_* Y
    $$
    together with the family of maps
    $$
        \bigl(X_{C_k}\longrightarrow X_{\partial C_k}\times_{Y_{\partial C_k}} Y_{C_k}\bigr)_{k\in\mathbb{N}}.
    $$
    Define $\overline{\mathcal{W}}$ to be the preimage under $\Phi$ of the class $\mathcal{W}'$ of pairs consisting of an equivalence of operads and a tuple of weak equivalences in $\sset$ for all $k\in\mathbb{N}$.  
    By Propositions~\ref{pr:description_ho} and~\ref{pr:char_fullyfaithful}, the restriction of $\overline{\mathcal{W}}$ to Segal spaces is precisely $\mathcal{W}_f$. This shows \rm(i). Since $\mathcal{W}'$ forms the class of weak equivalences of a combinatorial model category, the full subcategory on this class is accessible, and due to $\Phi$ preserving filtered colimits, so is its preimage by \cite[Corollary~A.2.6.5]{lurie2008highertopostheory}, which shows \rm(ii).
\end{proof}

\begin{prop}\label{pr:fact_J-fib}
    Let $X$ be a dendroidal Segal space.  
    In the factorization
    $$
        X\xrightarrow{w} \expd{X}{i_!NI[1]}\xrightarrow{p} X\times X
    $$
    obtained by applying $\expd{X}{i_!-}$ to the factorization
    $$
        N[0]\amalg N[0]\;\longrightarrow\; NI[1]\;\longrightarrow\; N[0],
    $$
    the map $w$ is a Dwyer--Kan equivalence and $p$ is an isofibration.
\end{prop}

\begin{proof}
    Setting $Y=*$ in Proposition~\ref{pr:exp_isofib} immediately shows that $p$ is an isofibration.

    To see that $w$ is a Dwyer--Kan equivalence, we adapt Rezk’s original argument from \cite[Lemma~13.9]{Rezk_2001}.  
    Using the identification
    $$
        i^*(\expd{X}{i_!-})\cong (i^*X)^{-}
    $$
    and the fact that $i^*X$ is a simplicial Segal space, Rezk’s proof shows that $i^*w$ is a Dwyer--Kan equivalence. It therefore remains only to verify homotopical fully faithfulness for corollas $C_k$ with $k\neq 1$.

    The argument Rezk gives for the case $k=1$ applies verbatim for all $k$. Consider the diagram
    $$
        \begin{tikzcd}
            X_{C_k} \arrow[r, "w_{C_k}"] \arrow[d] &
            (\expd{X}{i_!NI[1]})_{C_k} \arrow[r, "\iota^*"] \arrow[d] &
            (\expd{X}{i_!N[1]})_{C_k} \arrow[d] \\
            X_{\partial C_k} \arrow[r, "w_{\partial C_k}"] &
            (\expd{X}{i_!NI[1]})_{\partial C_k} \arrow[r, "\iota^*"] &
            (\expd{X}{i_!N[1]})_{\partial C_k}.
        \end{tikzcd}
    $$
    By Proposition~\ref{pr:char_fullyfaithful}, it suffices to show that the left square is a homotopy pullback.  
    The top horizontal map in the right square can be identified, via simplicial adjunction, with
    $$
        \Map{NI[1]}{i^*(\expd{X}{C_k})}\longrightarrow
        \Map{N[1]}{i^*(\expd{X}{C_k})},
    $$
    and similarly for the bottom map with $\partial C_k$, which are homotopy monomorphisms by \cite[Theorem~6.2]{Rezk_2001}, since $i^*(\expd{X}{C_k})$ is a simplicial Segal space by Proposition~\ref{pr:Segal_tensor}.

    The outer rectangle is isomorphic to
    $$
        \begin{tikzcd}
            X_{C_k} \arrow[r] \arrow[d] &
            X_{C_1\circ C_k}\times_{X_{C_k}}
            X_{C_k\circ(C_1,\ldots,C_1)} \arrow[d] \\
            \prod_{\partial C_k} X_\eta \arrow[r] &
            \prod_{\partial C_k} X_{C_1},
        \end{tikzcd}
    $$    
    which is a homotopy pullback: both inner squares 
    $$
        \begin{tikzcd}
             X_{C_k} \arrow[r] \arrow[d] & X_{C_1\circ C_k} \arrow[d] &&&  X_{C_k} \arrow[r] \arrow[d] & X_{C_k\circ (C_1,\ldots, C_1)} \arrow[d]\\
             X_{\eta} \arrow[r] & X_{C_1} &&& \prod_{k}X_{\eta} \arrow[r] & \prod_{k}X_{C_1}
        \end{tikzcd}
    $$
    are homotopy pullbacks by the Segal condition.  
    Thus $\iota^*\circ w_{C_k}$ is a weak equivalence on fibers, and since $\iota^*$ is a homotopy monomorphism on fibers, $w_{C_k}$ is a weak equivalence on fibers as well.
\end{proof}

We have now verified condition~\ref{condIII}.  
We next verify condition~\ref{condIV} by proving that (ii) implies (i) in the following characterization:

\begin{prop}\label{pr:char_I-fib_bw_Segal}
    Let $f\colon X\to Y$ be a map of dendroidal Segal spaces.  
    The following are equivalent:
    \begin{enumerate}
        \item[{\rm (i)}] \label{it:I_fib} $f$ is an $\mathcal{I}$-fibration.
        \item[{\rm (ii)}] \label{it:isofib_&DKeq} $f$ is an isofibration and a Dwyer–Kan equivalence.
        \item[{\rm (iii)}] $f$ is a Reedy fibration, homotopically fully faithful, and the induced map $X_{\eta,0}\to Y_{\eta,0}$ is surjective.
    \end{enumerate}
\end{prop}

\begin{proof}
    Proposition~\ref{pr:char_I-fibs} together with Proposition~\ref{pr:char_fullyfaithful} shows that condition~(i) is equivalent to condition~(iii).  
    For any Reedy fibration between Segal spaces, being an isofibration and a Dwyer--Kan equivalence is equivalent to being homotopically fully faithful and inducing an equivalence of operads that is surjective on objects.  
    This in turn is equivalent to homotopical fully faithfulness together with surjectivity of $f_\eta$ on vertices.
\end{proof}

Finally, to verify condition~\ref{condV}, we follow the strategy of \cite{moser2024}.  
We begin with a technical result generalizing \cite[Proposition~3.18]{moser2024}.

\begin{prop}\label{or:oullback_cosk_left_Quillen}
    Let $p\colon K\to L$ be a surjective Kan fibration between Kan complexes.  
    Then the functor
    $$
        \cosk_\eta(p)^*\colon {\dspsegal}_{/\cosk_\eta L}\longrightarrow {\dspsegal}_{/\cosk_\eta K}
    $$
    is left Quillen.
\end{prop}

\begin{proof}
    Pullbacks always preserve monomorphisms, and by \cite[Proposition~3.28]{Heuts2022}, they also preserve normal monomorphisms of dendroidal sets.  
    Since pullbacks can be computed termwise, either first in $\Delta$ and then in $\Omega$, or vice versa, they preserve normal monomorphisms of dendroidal spaces as well (these are precisely the termwise normal monos in $\Delta$). The same argument shows that pullbacks preserve termwise weak equivalences, using right properness of $\sset$.  
    Hence
    $$
        \cosk_\eta(p)^*\colon {\dsp_{\mathrm{Reedy}}}_{/\cosk_\eta L}\longrightarrow {\dspsegal}_{/\cosk_\eta K}
    $$
    is a left Quillen functor.

    To conclude, \cite[Proposition~3.3.18]{Hirschhorn2009} reduces the problem to showing that for every tree $T$ and every inner edge $e\colon\eta\to T$, the maps
    $$
        \begin{tikzcd}
            \Lambda^e[T] \arrow[rd] \arrow[rr, hook] && T \arrow[ld] \\
            & \cosk_\eta L
        \end{tikzcd}
    $$
    in $\dsp_{/\cosk_\eta L}$ are sent to trivial cofibrations in ${\dspsegal}_{/\cosk_\eta K}$.
    Consider the diagram
    $$
        \begin{tikzcd}
            Q & P & \cosk_\eta K \\
            \Lambda^e[T] & T & \cosk_\eta L.
            \arrow[from=1-1, to=1-2]
            \arrow[from=1-1, to=2-1]
            \arrow["\lrcorner"{anchor=center, pos=0.125}, draw=none, from=1-1, to=2-2]
            \arrow[from=1-2, to=1-3]
            \arrow[from=1-2, to=2-2]
            \arrow["\lrcorner"{anchor=center, pos=0.125}, draw=none, from=1-2, to=2-3]
            \arrow["\cosk_\eta(p)", from=1-3, to=2-3]
            \arrow[from=2-1, to=2-2]
            \arrow[from=2-2, to=2-3]
        \end{tikzcd}
    $$

    \emph{Case 1: $K$ and $L$ discrete}. Then $\cosk_\eta K \cong N_d(\mathrm{Comm}[K])$, the dendroidal nerve of the operad with set of colours $K$ and a unique operation on every possible multi-hom set, and similarly for $L$. Moreover, the map $\cosk_\eta(p)$ is the nerve of the unique map of operads which equals $p$ on colours.
 
    For each edge $\alpha:\eta\to T$ of the tree $T$, define $K_\alpha$ to be the fiber in $K$ over the image of $\alpha$ in $L$.
    Then, $P$ is itself isomorphic to the nerve of an operad, whose set of colours is $\coprod_{\alpha:\eta\to T}K_\alpha$ and whose operations are given by subtrees $T'\subset T$ with the root and each leaf decorated by an element of $K_\alpha$ for the corresponding $\alpha:\eta\to T$. These decorations give the target and the source respectively. In the case where the subtree $T'=\eta$, we label it twice, once as a leaf and once as a root. Composition is given by grafting and forgetting the intermediate label. The map $P\to T$ sends each $K_\alpha$ to the corresponding $\alpha$ and forgets all labels.

    A general dendrex in $P$ of shape $S$ is given by a map of trees $\beta:S\to T$ and for every edge $\gamma:\eta\to S$ an element of $K_{\beta\gamma}$.
    Such a dendrex is degenerate if and only if there is a unary vertex in $S$ such that its two edges get mapped to the same edge $\alpha:\eta\to T$ and moreover they are labeled by the same element $x\in K_\alpha$. It is not contained in $Q$ if and only the preimage under $\beta$ of each edge in $T$ is non empty, except for possibly that of $e$, that is
    \begin{equation*}
        \text{Edges}(T)\setminus\{e\}\subset \beta(\text{Edges}(S))
    \end{equation*}
    Note that for any $\beta:S\to T$ the preimage under $\beta$ of any edge of $T$ is either empty or a linear subtree of $S$. For a non-degenerate dendrex in $P$ not contained in $Q$ as above, define its \emph{pivot} as the label of the first edge in the linear tree $\beta^{-1}(e)$. Note that such a dendrex might not have a pivot, namely when $\beta^{-1}(e)$ is empty.
    
    Now fix an element $x\in K_e$. We say that a non-degenerate dendrex which is not contained in $Q$ is \emph{bounding} if it has $x$ as a pivot, and \emph{bounded} otherwise.
    Then every bounded dendrex is the inner face of a unique bounding dendrex, namely the one obtained from introducing a new edge at the start of $\beta^{-1}(e)$ with $x$ as a label. Moreover, given a bounding dendrex, all its faces are either smaller dendrices in $Q$ or smaller bounding dendrices, except for the face which deletes the pivot, which gives a bounded dendrex. This, together with the fact that $P$ is normal, means that we can write the inclusion $Q\xhookrightarrow{} P$ as a sequence of pushouts of inner horn inclusions, so that it is a trivial cofibration in ${\dspsegal}_{/\text{cosk}_\eta K}$.

    \emph{Case 2: general $K$ and $L$}.  
    Since \(p\colon K\to L\) is surjective, the discrete argument applies levelwise to the maps $K_n\to L_n$, obtaining a simplicial diagram
    \begin{equation}
        Q_\bullet = \Lambda^e[T]\times_{\cosk_\eta(L_\bullet)} \cosk_\eta(K_\bullet)
        \longrightarrow
        P_\bullet = T\times_{\cosk_\eta(L_\bullet)} \cosk_\eta(K_\bullet)
    \label{eq:simplicial_diagram}
    \end{equation}
    in $\dsp$, each \(Q_n\to P_n\) a weak equivalence in $\dspsegal$.  
    Since all $P_n$ and $Q_n$ are normal discrete dendroidal spaces, the simplicial objects $P_\bullet$ and $Q_\bullet$ are Reedy cofibrant in $(\dsp)^{\Delta^{\mathrm{op}}}$. Indeed, regarding this category as the category of functors $\Delta^{\text{op}}\times \Delta^{\text{op}}\to \dset$, the cofibrations are precisely the monomorphisms which are normal monomorphisms of dendroidal sets at each bisimplicial bidegree. Thus by \cite[Corollary~18.4.13]{Hirschhorn2009}, if we apply the geometric realization
    \begin{equation*}
        (\dsp)^{\Delta^{\text{op}}}\to \dsp
    \end{equation*}
    to \eqref{eq:simplicial_diagram}, we obtain a weak equivalence in $\dspsegal$. But in our case the geometric realization is the diagonal, which gives precisely the map $Q\hookrightarrow P$. 
\end{proof}

\begin{prop}
    Let \(f\colon X\to Y\) be an \(\mathcal{I}\)-fibration.  
    Then \(f\) admits a \(\mathcal{J}\)-fibrant replacement that is a Dwyer--Kan equivalence.
\end{prop}

\begin{proof}
    Apply the small object argument to obtain a $\mathcal{J}_0$-fibrant replacement of $f$:
    \[
        \begin{tikzcd}
            X \arrow[r, "j_X"] \arrow[d, "f"] & \widetilde{X} \arrow[d, "\widetilde{f}"] \\
            Y \arrow[r, "j_Y"] & \widetilde{Y}.
        \end{tikzcd}
    \]
    The right vertical map is a $\mathcal{J}_0$-fibration and the horizontal maps are $\mathcal{J}_0$-cofibrations with fibrant codomain. Hence, by Proposition~\ref{pr:J_0_replacement}, these are in fact $\mathcal{J}$-fibrant. Therefore $\widetilde{f}$ is a $\mathcal{J}$-fibrant replacement of $f$.

    Since $j_X$ and $j_Y$ are weak equivalences at $\eta$ (Proposition~\ref{pr:J0_trivial_at_eta}), and $f_\eta$ is surjective (Proposition~\ref{pr:char_I-fibs}), it follows that $\pi_0(\widetilde{f}_\eta)$ is surjective.  
    As $\widetilde{f}$ is a $\mathcal{J}_0$-fibration, $\widetilde{f}_\eta$ is a Kan fibration and hence it is surjective on vertices. Therefore $\ho(\widetilde{f})$ is essentially surjective.

    For fully faithfulness, by Remark~\ref{rm:alt_fullyfaithful} it suffices to show that
    $$
        \widetilde{\varphi}\colon 
        \widetilde{X}\longrightarrow 
        \widetilde{Y}\times_{\cosk_\eta\widetilde{Y}_\eta}\cosk_\eta\widetilde{X}_\eta =:Q
    $$
    is a weak equivalence in $\dsp_{\mathrm{Segal}}$. Consider the diagram
    $$
        \begin{tikzcd}
            X \arrow[r] \arrow[d, "f"] &
            \cosk_\eta X_\eta \arrow[r, "\cosk_\eta (j_X)_\eta"] \arrow[d] &
            \cosk_\eta \widetilde{X}_\eta \arrow[d] \\
            Y \arrow[r] &
            \cosk_\eta Y_\eta \arrow[r, "\cosk_\eta (j_Y)_\eta"] &
            \cosk_\eta \widetilde{Y}_\eta.
        \end{tikzcd}
    $$
    By applying Remark \ref{rm:fullyfaithful_reedy_fib_case} to $f$, we see that the left square is a homotopy pullback at every tree $T$ (using right properness of $\sset$ and the fact that the vertical maps are termwise fibrations). By Proposition~\ref{pr:J0_trivial_at_eta}, the maps $j_X$ and $j_Y$ are weak equivalences at $\eta$, so the horizontal maps in the right side square are weak equivalences at every tree $T$, as they are $n$-fold products of $(j_X)_\eta$ and $(j_Y)_\eta$ for $n$ the number of edges of $T$. Consequently the right side square is also a homotopy pullback at every tree, and so is the composite, so the induced map 
    $$
    \varphi:X\longrightarrow Y \times_{\text{cosk}_\eta \Tilde{Y}_\eta} \text{cosk}_\eta \Tilde{X}_\eta =P
    $$ 
    is a Reedy weak equivalence (again because of the right properness of $\sset$ and the vertical termwise fibrations).
    Now consider
    \[
        \begin{tikzcd}
            X \arrow[r] \arrow[d, "\varphi"'] &
            \widetilde{X} \arrow[d, "\widetilde{\varphi}"] \\
            P \arrow[r] & Q.
        \end{tikzcd}
    \]
    We have just shown that $\varphi\colon X\to P$ is a Reedy weak equivalence and hence a weak equivalence in $\dspsegal$.  
    Since $X\to\widetilde{X}$ and $Y\to\widetilde{Y}$ are $\mathcal{J}_0$-cofibrations, they are Segal weak equivalences. By Proposition~\ref{or:oullback_cosk_left_Quillen}, the map $P\to Q$ is also a Segal weak equivalence. 
    Applying 2-out-of-3 twice yields that $\widetilde{\varphi}\colon \widetilde{X}\to Q$ is a weak equivalence in $\dspsegal$.
\end{proof}

This establishes condition~\ref{condV} and therefore completes the proof of the existence of the model structure in Theorem~\ref{thm:main}.

\section{Comparison with other model structures}

We now study the relationship between the model category $\dspop$ and several other homotopical frameworks, namely simplicial Segal spaces, complete dendroidal Segal spaces, and Segal operads.  
To this end, we will make use of the following general criterion.

\begin{lem}\label{lm:pseudo_gen}
    Let $F\colon C \rightleftarrows D\colon G$ be an adjunction between model categories such that $F$ preserves cofibrations.  
    Suppose that there exists a set $J$ of cofibrations in $C$ with the following properties:
    \begin{itemize}
        \item[{\rm (i)}] An object $X\in C$ is fibrant if and only if it has the right lifting property with respect to $J$.
        \item[{\rm (ii)}] A map $f\colon X\to Y$ between fibrant objects is a fibration if and only if it has the right lifting property with respect to $J$.
    \end{itemize}
    If moreover $F$ sends every map in $J$ to a weak equivalence in $D$, then $F$ is a left Quillen functor.
\end{lem}

\begin{proof}
    Let $p\colon X\to Y$ be a fibration between fibrant objects in $D$.  
    Since $F$ preserves cofibrations, \cite[Proposition~E.2.14]{joyal} implies that it suffices to show that the right adjoint $G$ sends $p$ to a fibration in $C$.

    Because $F(J)$ consists of trivial cofibrations, the object $X$ has the right lifting property with respect to $F(J)$.  
    By adjunction, this implies that $G(X)$ has the right lifting property with respect to $J$, hence is fibrant by hypothesis. The same argument applies to $G(Y)$.

    Moreover, since $p$ lifts against $F(J)$, adjunction again shows that $G(p)$ lifts against $J$. As $G(p)$ is a map between fibrant objects with the right lifting property against $J$, it follows that $G(p)$ is a fibration in $C$.
\end{proof}

Proposition~\ref{pr:J-fib_iff_isofib} shows that the hypotheses of Lemma~\ref{lm:pseudo_gen} are satisfied for $J=\mathcal{J}$.  
Consequently, this criterion will allow us to verify that several left adjoints out of $\dspop$ are left Quillen functors.

\subsection{Simplicial Segal spaces}

We begin by comparing the model structure $\dspop$ with the model structure $\sspcat$ introduced in \cite{moser2024}.  
The latter is obtained by applying Theorem~\ref{th:fibrantly-induced} with the sets $\mathcal{I}_{\mathrm{Cat}}$ and $\mathcal{J}_{\mathrm{Cat}}$ defined in \cite[Notation~3.5]{moser2024} and \cite[Notation~3.9]{moser2024} (where they appear without the subscript).  
In terms of our generating sets $\mathcal{I}$ and $\mathcal{J}$, these can be succinctly described as
$$
    \mathcal{I}_{\mathrm{Cat}} = i_!^{-1}(\mathcal{I}),
    \qquad
    \mathcal{J}_{\mathrm{Cat}} = i_!^{-1}(\mathcal{J}).
$$

We will start by showing that the model structure $\sspcat$ can in fact be recovered from $\dspop$.
Recall that the category $\ssp$ is canonically isomorphic to the slice category $\dsp_{/\eta}$, with the functor $i_!$ corresponding to the slice projection.  
It follows that $\dspop$ induces a model structure on $\ssp$, which we denote by $\ssp_{\mathrm{slice}}$.

\begin{thm}
    The model structures $\sspcat$ and $\ssp_{\mathrm{slice}}$ coincide.  
    In particular, the adjunction
    \begin{equation*}
        i_! \colon \sspcat \rightleftarrows \dspop \colon i^*
    \end{equation*}
    is a Quillen adjunction.
\end{thm}

\begin{proof}
    This follows immediately from the definition of the induced slice model structure together with the equalities
    $i_!^{-1}(\mathcal{I}) = \mathcal{I}_{\mathrm{Cat}}$ and
    $i_!^{-1}(\mathcal{J}) = \mathcal{J}_{\mathrm{Cat}}$.  
    Hence both model structures have the same classes of cofibrations and fibrant objects, which uniquely determine the model structure.
\end{proof}

\begin{prop}\label{pr:comp_with_Cat}
    The dendroidal tensor
    \begin{equation*}
        - \otimes i_! - \colon \dsp \times \ssp \longrightarrow \dsp
    \end{equation*}
    is a left Quillen bifunctor with respect to the model structures $\dsp_{\mathrm{Op}}$ and $\ssp_{\mathrm{Cat}}$.
\end{prop}

\begin{proof}
    Let $f$ and $g$ be cofibrations in $\dsp_{\mathrm{Op}}$ and $\ssp_{\mathrm{Cat}}$, respectively.  
    In particular, they are Reedy cofibrations, so by Proposition~\ref{pr:Reedy_tensor} the pushout--product
    $f \boxtimes i_! g$ is a normal monomorphism.  
    Since $(A \otimes B)_\eta = A_\eta \times B_\eta$, we have
    $$
        (f \boxtimes i_! g)_\eta = f_\eta \,\square\, i_! g_\eta .
    $$
    By Proposition~\ref{pr:ext_by_set_char}, the class of extensions by sets is generated by the trivial cofibrations of simplicial sets (which are closed under $\square$) together with the map $\emptyset \to \std{0}$, which is the unit for $\square$.  
    Hence extensions by sets are closed under $\square$, and Proposition~\ref{pr:char_I-cof} implies that $f \boxtimes i_! g$ is a cofibration in $\dsp_{\mathrm{Op}}$.

    Assume now that either $f$ or $g$ is a trivial cofibration.  
    We must show that $f \boxtimes i_! g$ is again a trivial cofibration.  
    By Lemma~\ref{lm:pseudo_gen}, it suffices to consider the case where either $f$ or $g$ is an anodyne extension, that is, a $\mathcal{J}$-cofibration.  
    Since $\mathcal{J}$-cof is generated by $\mathcal{J}_0$-cof together with the map $i_!N[0] \to i_!NI[1]$, we treat these two cases separately.

    If either $f$ or $g$ lies in $\mathcal{J}_0$-cof, then $f \boxtimes i_! g$ belongs to $\mathcal{J}_0$-cof by Proposition~\ref{pr:Reedy_tensor}.  
    If $f = i_!N[0] \to i_!NI[1]$, then using that $i_!(X \times Y) = i_!X \otimes i_!Y$ and the inclusion $i_!(\mathcal{J}_{\mathrm{Cat}}) \subset \mathcal{J}$, the claim reduces to the simplicial case. This holds since $\ssp_{\mathrm{Cat}}$ is cartesian closed \cite[Proposition~5.3]{moser2024}.

    Finally, if $g = N[0] \to NI[1]$, then $f \boxtimes i_! g$ lifts against an isofibration of dendroidal Segal spaces $p \colon X \to Y$ if and only if $f$ lifts against the induced map
    $$
        p^{\boxtimes i_! g} \colon \expd{X}{i_!NI[1]} \longrightarrow \expd{Y}{i_!NI[1]} \times_Y X .
    $$
    By Proposition~\ref{pr:exp_isofib}, this map is an isofibration.  
    If we show that it is also a Dwyer--Kan equivalence, then Proposition~\ref{pr:char_I-fibs} implies that it is an $\mathcal{I}$-fibration, and hence lifts against $f$.

    The maps $\expd{X}{i_!NI[1]} \to X$ and $\expd{Y}{i_!NI[1]} \to Y$ are isofibrations by Proposition~\ref{pr:exp_isofib} and Dwyer--Kan equivalences by Proposition~\ref{pr:fact_J-fib} and the $2$-out-of-$3$ property.  
    Hence they are $\mathcal{I}$-fibrations by Proposition~\ref{pr:char_I-fibs}.  
    In the factorization
    $$
        \expd{X}{i_!NI[1]} \longrightarrow \expd{Y}{i_!NI[1]} \times_Y X \to X ,
    $$
    the composite is therefore an $\mathcal{I}$-fibration, and the right-hand map, being a pullback of one, is again a Dwyer--Kan equivalence. The claim follows by another application of $2$-out-of-$3$.
\end{proof}

As in Theorem~\ref{th:Segal_ho_enriched}, we can now use this to establish a homotopical enrichment.

\begin{thm}
    The model structure $\dspop$ is homotopically $\sspcat$-enriched in the sense of \cite{Heuts2015}.
\end{thm}

\begin{proof}
    The first axiom follows from Proposition~\ref{pr:comp_with_Cat}.  
    For the second, the proof of Theorem~\ref{th:Segal_ho_enriched} shows that for all cofibrations
    $f,g \in \sspsegal$ and $h \in \dspsegal$, the map $\phi_{f,g,h}$ is a $\mathcal{J}_0$-cofibration, hence a trivial cofibration in $\dspop$.  
    Since the cofibrations in $\sspcat$ and $\dspop$ are contained in those of their respective Segal model structures, the result follows.
\end{proof}

This completes the proof of the final statement of Theorem~\ref{thm:main_existence}.

\subsection{Complete dendroidal Segal spaces}

We now compare the model structure $\dspop$ with the complete dendroidal Segal space model structure introduced by Cisinski and Moerdijk in \cite{Cisinski_Moerdijk_2}, which generalizes Rezk’s complete Segal space model structure from \cite{Rezk_2001}.

\begin{thm}
    There exists a left proper simplicial model structure on the category $\dsp$, obtained as a left Bousfield localization of $\dspsegal$ with respect to the morphism $i_!N[0]\to i_!NI[1]$.
    This model structure satisfies the following properties:
    \begin{itemize}
        \item[{\rm (i)}] The cofibrations are the normal monomorphisms.
        \item[{\rm (ii)}] The fibrant objects are the complete dendroidal Segal spaces.
        \item[{\rm (iii)}] The fibrations (respectively, weak equivalences) between fibrant objects are the Reedy fibrations (respectively, Reedy weak equivalences).
        \item[{\rm (iv)}] More generally, the weak equivalences between dendroidal Segal spaces are the Dwyer--Kan equivalences.
    \end{itemize}
    We denote this model structure by $\dspcss$.
\end{thm}

\begin{proof}
    This follows from standard results on left Bousfield localizations, except for the last statement, which is proved in \cite[Theorem~12.36]{Heuts2022}.
\end{proof}

The model structures $\dspop$ and $\dspcss$ turn out to be very closely related.  
As is evident from their definitions, the class of cofibrations in $\dspcss$ contains that of $\dspop$.  
More surprisingly, the two model structures have exactly the same weak equivalences.

\begin{lem}\label{lm:weak_eq_are_the_same}
    Let $f \colon X \to Y$ be a morphism in $\dsp$.  
    Then $f$ is a weak equivalence in $\dspop$ if and only if it is a weak equivalence in $\dspcss$.
\end{lem}

\begin{proof}
    Apply the small object argument to obtain a $\mathcal{J}$-fibrant replacement
    \[
        j_Y \colon Y \longrightarrow \tilde{Y},
    \]
    and then again to factor the composite $X \to Y \to \tilde{Y}$ as a $\mathcal{J}$-cofibration
    \[
        j_X \colon X \longrightarrow \tilde{X}
    \]
    followed by a $\mathcal{J}$-fibration
    \[
        \tilde{f} \colon \tilde{X} \longrightarrow \tilde{Y}.
    \]
    Since every map in $\mathcal{J}$ is a trivial cofibration in $\dspcss$, both $j_X$ and $j_Y$ are weak equivalences in $\dspcss$ (and of course also in $\dspop$).  
    Consequently, $f$ is a weak equivalence in $\dspcss$ (respectively, in $\dspop$) if and only if $\tilde{f}$ is.
    Finally, since $\tilde{X}$ and $\tilde{Y}$ are dendroidal Segal spaces, the map $\tilde{f}$ is a weak equivalence in $\dspcss$ if and only if it is a Dwyer--Kan equivalence (i.e. a weak equivalence in $\dspop$), by \cite[Theorem~12.36]{Heuts2022}.
\end{proof}

\begin{rem}
    An analogous statement holds in the simplicial setting. Although \cite{moser2024} establishes only one direction, the full equivalence appears not to have been explicitly noted there. As we will see, having the equivalence in both directions allows for more streamlined arguments than those in \cite{moser2024}.
\end{rem}

The situation can thus be summarized as follows: the model structures $\dspop$ and $\dspcss$ have the same weak equivalences; the cofibrations of $\dspcss$ contain those of $\dspop$; and the fibrations of $\dspop$ contain those of $\dspcss$.  
This immediately yields the following result.

\begin{thm}\label{thm:main_quilleneq}
    The identity adjunction
    $$
        \id{!} \colon \dspop \rightleftarrows \dspcss \colon \id{}^*
    $$
    is a Quillen equivalence.
\end{thm}

\begin{proof}
    By the discussion above, the functor $\id{!}$ preserves cofibrations and trivial cofibrations, and therefore defines a Quillen adjunction.  
    Lemma~\ref{lm:weak_eq_are_the_same} then implies that this Quillen adjunction is in fact a Quillen equivalence.
\end{proof}

Lemma~\ref{lm:weak_eq_are_the_same} has a further immediate consequence.  
Since left and right properness depend only on the class of weak equivalences, and since $\dspcss$ is left proper as a left Bousfield localization of the left proper model structure $\dsp_{\mathrm{Reedy}}$, we obtain the following.

\begin{thm}
    The model structure $\dspop$ is left proper. $\hfill\qed$
\end{thm}

\subsection{Operads}

Recall that the category $\op$ admits the folk model structure, in which the weak equivalences are the equivalences of operads and the fibrations are the isofibrations.

As explained in the preliminaries, there is an adjunction
$$
    \tau_d:\dsp\rightleftarrows \op:N_d.
$$
However, the dendroidal nerve functor is not right Quillen with respect to the model structure $\dspcss$. Indeed, the dendroidal nerve of an operad $\mathcal{O}$ is complete if and only if every isomorphism in $\mathcal{O}$ is an identity, a condition which fails for most operads (for instance, for $i_!NI[1]$).

We now show that this issue disappears when working with the model structure $\dspop$, so that the adjunction above becomes a genuine Quillen adjunction.

\begin{thm}
    The adjunction
    $$
        \tau_d:\dspop\rightleftarrows \op:N_d
    $$
    is a Quillen adjunction.
\end{thm}

\begin{proof}
    The dendroidal nerve functor sends every operad to a dendroidal Segal space and sends isofibrations of operads to isofibrations of dendroidal Segal spaces. Moreover, it sends equivalences of operads to Dwyer--Kan equivalences. Consequently, $N_d$ preserves fibrations and trivial fibrations, and is therefore a right Quillen functor.
\end{proof}

\subsection{Segal operads}

We now recall from \ref{sec:segal_preop} the category $\preop$ of Segal pre-operads, as well as the triple of adjoints

$$
\begin{tikzcd}
\preop \arrow[r, "\gamma^*"] & \dsp \arrow[l, "\gamma_!"', bend right, shift right=1] \arrow[l, "\gamma_*", bend left, shift left=1]
\end{tikzcd}
$$

\begin{defn}
    A \emph{Segal operad} is a Segal pre-operad $X$ such that for every tree $T$, the canonical map
    \begin{equation*}
        X_T \longrightarrow X_{\spn{T}}
    \end{equation*}
    is a trivial fibration of simplicial sets.
    A \emph{Reedy fibrant Segal operad} is a Segal pre-operad $X$ such that $\gamma^*X$ is a Segal space.
\end{defn}

In \cite{Cisinski_Moerdijk_2}, Cisinski and Moerdijk show that the Segal operads arise as the fibrant objects of a model structure on $\preop$.

\begin{thm}\label{th:model_str_P(Op)}
    There exists a left proper, cofibrantly generated model structure on the category $\preop$, which we denote by $\preop_{\mathrm{Segal}}$, with the following properties:
    \begin{itemize}
        \item[{\rm (i)}] The fibrant objects are the Reedy fibrant Segal operads.
        \item[{\rm (ii)}] The weak equivalences are precisely the maps which are weak equivalences in $\dspcss$.
        \item[{\rm (iii)}] The cofibrations are the normal monomorphisms.
    \end{itemize}
    Moreover, the adjoint pair $\gamma^* \dashv \gamma_*$ induces a Quillen equivalence
    $$
        \gamma^*\colon \preop_{\text{Segal}} \rightleftarrows \dspcss \colon \gamma_*.
    $$
\end{thm}

\begin{proof}
    This is proved in \cite[Theorems~8.13, 8.15, 8.17]{Cisinski_Moerdijk_2}.
\end{proof}

However, the adjunction
$$
    \gamma_!\colon \dsp \rightleftarrows \preop \colon \gamma^*
$$
is not a Quillen adjunction with respect to the above model structures. This is yet another respect in which the model structure $\dspop$ exhibits better formal properties than $\dspcss$, as shown by the following result.

\begin{thm}
    The adjunctions $\gamma_! \dashv \gamma^*$ and $\gamma^* \dashv \gamma_*$ induce Quillen equivalences
    $$
        \gamma_!\colon \dspop \rightleftarrows \preop_{\mathrm{Segal}} \colon \gamma^*
    $$
    and
    $$
        \gamma^*\colon \preop_{\mathrm{Segal}} \rightleftarrows \dspop \colon \gamma_*.
    $$
\end{thm}

\begin{proof}
    We begin with the adjunction $\gamma_! \dashv \gamma^*$.
    Since all extensions by sets are injective on connected components, Proposition~\ref{pr:char_I-cof} together with \cite[Lemma~7.4]{Cisinski_Moerdijk_2} implies that $\gamma_!$ preserves cofibrations.
    To see that $\gamma_!$ preserves trivial cofibrations, it suffices by Lemma~\ref{lm:pseudo_gen} to show that it sends every map in $\mathcal{J}$ to a weak equivalence in $\preop_{\mathrm{Segal}}$.

    Unravelling the definitions, this amounts to showing that for $f \in \mathcal{J}$, the map $\gamma^*\gamma_! f$ is a weak equivalence in $\dspcss$.
    If $f$ is the map $\eta \to i_!NI[1]$, then $f$ is a Dwyer--Kan equivalence and satisfies $f = \gamma^*\gamma_! f$. If $f$ is of the form
    $$
        \pop{\cdot}{\partial T}{T}{\horn{k}{n}}{\std{n}}
    $$
    with $T = \eta$, a straightforward computation shows that $\gamma^*\gamma_! f$ is isomorphic to $\id{\eta}$. In all remaining cases, $f$ is an isomorphism at $\eta$, which implies that the naturality square
    $$
        \begin{tikzcd}
            A \arrow[d, "f"'] \arrow[r] & \gamma^*\gamma_! A \arrow[d, "\gamma^*\gamma_! f"] \\
            B \arrow[r]                 & \gamma^*\gamma_! B
        \end{tikzcd}
   $$
    is a pushout square. Since $f$ is a trivial cofibration in $\dspcss$, this completes the argument.

    To conclude that $(\gamma_!,\gamma^*)$ is a Quillen equivalence, consider the diagram of right Quillen functors induced by the counit $\varepsilon\colon \gamma_*\gamma^* \to \id{\dsp}$ of the adjunction $\gamma^* \dashv \gamma_*$:
    $$
        \begin{tikzcd}
        	& {\preop_{\text{Segal}}} \\
        	\dspcss && \dspop
        	\arrow["{\gamma^*}", from=1-2, to=2-3]
        	\arrow["{\gamma_*}", from=2-1, to=1-2]
        	\arrow[""{name=0, anchor=center, inner sep=0}, "{\id{}^{*}}"', from=2-1, to=2-3]
        	\arrow["\varepsilon", shorten >=3pt, Rightarrow, from=1-2, to=0]
        \end{tikzcd}
    $$
    By Example~\ref{ex:gammaXtoX_DK}, the counit is a weak equivalence on fibrant objects of $\dspcss$, and hence the right derived functors of $\gamma^*\gamma_*$ and $\id{}^*$ agree.
    It follows that $\gamma^*\gamma_*$ is a Quillen equivalence, and therefore so is $\gamma^*$ by Theorem~\ref{th:model_str_P(Op)} and 2-out-of-3.

    Finally, for the adjunction $\gamma^* \dashv \gamma_*$, note that normal monomorphisms of Segal pre-operads are maps of sets at $\eta$, and thus Proposition~\ref{pr:char_I-cof} implies that $\gamma^*$ preserves cofibrations.
    Moreover, by Theorem~\ref{th:model_str_P(Op)} and Lemma~\ref{lm:weak_eq_are_the_same}, the functor $\gamma^*$ preserves and reflects weak equivalences.
    Applying 2-out-of-3 to the diagram of left Quillen functors
    $$
        \begin{tikzcd}
            & \dspop \arrow[rd, "\id{!}"] & \\
            \preop_{\text{Segal}} \arrow[ru, "\gamma^*"] \arrow[rr, "\gamma^*"] & & \dspcss
        \end{tikzcd}
    $$
    we conclude that $\gamma^*$ is a left Quillen equivalence.
\end{proof}

\begin{cor}
    The model structure $\preop_{\text{Segal}}$ is both left- and right-induced from $\dspop$ along the functor $\gamma^*$. $\hfill\qed$
\end{cor}

\begin{proof}
    This follows formally from the fact that $\gamma^*$ is fully faithful and both a left and right Quillen equivalence; see \cite[Lemma~4.15]{moser2024}.
\end{proof}

As a consequence, we obtain the following characterization.

\begin{cor}
    The fibrations in $\preop_{\text{Segal}}$ between fibrant objects (i.e.\ Reedy fibrant Segal operads) are precisely the isofibrations. $\hfill\qed$
\end{cor}

\section{An application to homotopy limits and colimits}

As an application of the results above, we show that for sufficiently well-behaved diagrams of dendroidal Segal spaces, the completion step can be omitted when computing homotopy limits in $\dspcss$. This follows from the next result.

\begin{thm}\label{th:bestofbothworlds}
    Let $F\colon J \to \dsp$ be a diagram and let $\alpha\colon X \Rightarrow F$ be a cone from an object $X$ to $F$. Then $\alpha$ is a homotopy limit cone in $\dspop$ if and only if it is a homotopy limit cone in $\dspcss$.
    Dually, a cocone under $F$ is a homotopy colimit in $\dspop$ if and only if it is a homotopy colimit in $\dspcss$.
\end{thm}

\begin{proof}
    This is an immediate consequence of Lemma~\ref{lm:weak_eq_are_the_same}, together with the fact that the notion of homotopy limit (and dually, homotopy colimit) depends only on the class of weak equivalences.
\end{proof}

Note that since $\dspop$ has more fibrations than $\dspcss$, the condition of being injectively fibrant in $\dspop$ is weaker than in $\dspcss$, while the converse holds for cofibrations. In this sense, Theorem~\ref{th:bestofbothworlds} shows that, when computing homotopy limits and colimits, we truly obtain the best of both worlds: it suffices to perform the weaker of the two replacements. In particular, any pullback square
$$
    \begin{tikzcd}
        X' \arrow[r] \arrow[d] & X \arrow[d, "f"] \\
        Y' \arrow[r]           & Y
    \end{tikzcd}
$$
in which $X$, $Y$, and $Y'$ are dendroidal Segal spaces and $f$ is an isofibration is a homotopy pullback in $\dspcss$. Dually, any pushout square
\begin{equation*}
    \begin{tikzcd}
        A \arrow[r, "i"] \arrow[d] & B \arrow[d] \\
        C \arrow[r]               & D
    \end{tikzcd}
\end{equation*}
in which $A$, $B$, and $C$ are dendroidal spaces and $i$ is a normal monomorphism is a homotopy pushout in $\dspop$.

\bibliographystyle{abbrv}
\bibliography{bib}
\end{document}